\documentclass[11pt,a4paper]{article}

\usepackage{amssymb} 
\usepackage[intlimits]{amsmath}
\usepackage{amsfonts}
\usepackage{amsthm} 
\usepackage{mathrsfs} 

\usepackage{hyperref}
\hypersetup{colorlinks=true,linkcolor=RoyalPurple,citecolor=RoyalPurple}
\usepackage{cite}
\usepackage{graphicx}
\usepackage{geometry}
\usepackage{ulem}
\usepackage{color}
\usepackage[dvipsnames]{xcolor}

\let\oldsqrt\sqrt
\def\sqrt{\mathpalette\DHLhksqrt}
\def\DHLhksqrt#1#2{%
\setbox0=\hbox{$#1\oldsqrt{#2\,}$}\dimen0=\ht0
\advance\dimen0-0.2\ht0
\setbox2=\hbox{\vrule height\ht0 depth -\dimen0}%
{\box0\lower0.4pt\box2}}

\allowdisplaybreaks

\newcommand{\R}{\mathbb{R}} 
\newcommand{\N}{\mathbb{N}} 

\newcommand{\dist}{\textnormal{dist}} 
\newcommand{\essinf}{\textnormal{essinf}} 
\newcommand{\esssup}{\textnormal{esssup}} 

\newcommand{\cE}{{\mathcal E}}

\newcommand{\cH}{{\mathcal H}}

\newcommand{\cL}{{\mathcal L}}

\newcommand{\cO}{{\mathcal O}}

\renewcommand{\phi}{\varphi}
\newcommand{\eps}{\varepsilon}

\newcommand{\op}{\mathscr{L}}
\renewcommand{\div}{\textnormal{div}}
\newcommand{\loc}{\operatorname{loc}}

\usepackage[utf8]{inputenc}
\usepackage[capitalize]{cleveref}
\crefname{equation}{}{} 
\usepackage{enumitem}

\newtheorem{thm}{Theorem}[section]

\newtheorem{prop}[thm]{Proposition}
\newtheorem{lem}[thm]{Lemma}
\newtheorem*{theorem*}{Theorem}

\theoremstyle{definition}
\newtheorem{definition}[thm]{Definition}
\newtheorem{remark}[thm]{Remark}
\newtheorem*{remark*}{Remark}
\newtheorem{example}{Example}

\title{Continuity of solutions to equations with weakly singular nonlocal operators}
\author{Sven Jarohs\footnote{Institut f\"ur Mathematik, Goethe-Universit\"{a}t, Frankfurt, Robert-Mayer-Stra\ss e 10, D-60629 Frankfurt, jarohs@math.uni-frankfurt.de.}, 
\
Moritz Kassmann\footnote{Universit\"{a}t Bielefeld, Fakult\"{a}t f\"{u}r Mathematik, Universit\"{a}tsstra{\ss}e 25, D-33615 Bielefeld, moritz.kassmann@uni-bielefeld.de},
\
and
\
Tobias Weth\footnote{Institut f\"ur Mathematik, Goethe-Universit\"{a}t, Frankfurt, Robert-Mayer-Stra\ss e 10, D-60629 Frankfurt, weth@math.uni-frankfurt.de.}
}
\date{\today}

\begin{document}
\maketitle

\begin{abstract}
We prove boundedness and regularity estimates for weak solutions to a class of linear nonlocal equations involving integro-differential operators with almost no order of differentiability. In particular, we show that bounded weak solutions are continuous, and we provide a uniform a-priori estimates for the modulus of continuity. In contrast to earlier works, we allow the nonlocal operators to be highly anisotropic and weakly singular, and we allow the associated kernel functions to vanish close to the singularity.
\end{abstract}

{\footnotesize
\begin{center}

\end{center}
}

\setcounter{section}{0}
\section{Introduction}\label{sec:introduction}

The present paper is devoted to regularity properties of weak solutions to nonlocal equations of the type
\begin{equation}
\label{general-op-eq}
\op u= Wu + f \qquad \text{in $\Omega$}
\end{equation}
with measurable functions $W,f:\Omega\to\R$, which are driven by a nonlocal operators $\op $ of the form 
\begin{align}\label{def:op-new}
\op u(x)= \int_{\R^N} \big(u(x)-u(y) \big) K(x,y)\,dy \qquad (x\in \R^N) \,,
\end{align}
resp. its corresponding bilinear form
\begin{equation}\label{defi:bilinear}
(u,v)\mapsto \cE(u,v):=\iint_{\R^N\,\R^N}\big(u(x)-u(y)\big)v(x)K(x,y)\, dydx.
\end{equation}
Here, $\Omega \subset \R^N$ is an open set, and $K:\R^N \times \R^N \to[0,\infty]$ is a measurable function satisfying the translation invariant comparability condition
\begin{align} \label{eq:cond-K}\tag{K} 
\Lambda^{-1} j(y-x) \leq K(x,y) \leq \Lambda j(y-x)	\qquad (x,y \in \R^N,\, x\neq y).
\end{align}
for some $\Lambda \geq 1$ and some measurable function $j:\R^N\setminus \{0\} \to [0, \infty)$. Under this assumption, key properties of the bilinear form $\cE$ are governed by the function $j$, and therefore the regularity of weak solutions of \eqref{general-op-eq} is intimately related to the specific form of $j$. 
We also assume $W \in L^\infty(\Omega)$ in \eqref{general-op-eq}, and we stress that the special case $W \equiv \lambda$, $\lambda \in \R$, $f \equiv 0$ leads to the eigenvalue equation for the operator $\op$ in $\Omega$.
 
A challenging question, which received considerable attention in recent years, is to derive regularity estimates for solutions of \eqref{general-op-eq} depending on specific assumptions on the governing function $j$. It is worth noting that the regularity theory of \eqref{general-op-eq} is extremely well-developed in the case where $j(z)=|z|^{-N-2s}$ for some $s \in (0,1)$, in which $\cL$ coincides, up to a normalization constant, with the well-studied fractional Laplacian $(-\Delta)^s$, see e.g. \cite{AV19,G19} and the references therein.

The main purpose of the present paper is to provide boundedness and continuity estimates for solutions of \eqref{general-op-eq} while assuming as little as possible on the governing function $j$. We first introduce the following general assumption on $j$ with parameter $\gamma >0$:
\begin{equation}
  \int_{\R^N} \min\{1, |z|^\gamma\} \, j(z) dz < \infty. \tag{$A1_\gamma$} \label{eq:cond-A1}
\end{equation}
Condition \eqref{eq:cond-A1} with $\gamma=2$ is often called L\'{e}vy condition because the measure $j(z)\, dz$ then is known as a L\'{e}vy measure in probability theory. A consequence of  \eqref{eq:cond-A1} (for $\gamma\in(0,1)$) is that the expression $|\op u(x)|$ is finite for every measurable bounded function $u:\R^N \to \R$, which is Hölder-continuous of order greater than $\gamma$ in a neighborhood of the point $x \in \R^N$. In the present paper, however, we deal with the natural class of weak (or variational) solutions for which (\ref{general-op-eq}) does not have a pointwise meaning in general.

The notion of a weak solution arises very naturally in the case where the kernel function $K$ is symmetric. On the other hand, as noted in \cite{FKV15}, it can also be defined for nonsymmetric $K$ if its antisymmetric part is controlled by its symmetric part. More precisely, we consider, as in \cite{FKV15}, the symmetric and antisymmetric parts of $K$ given by 
$$
K_s(x,y):=\frac12 \big( K(x,y) + K(y,x) \big)\quad\text{and}\quad K_a(x,y) := \frac12 \big( K(x,y) - K(y,x) \big),
$$
for $x,y\in \R^N$, $x\neq y$, and we assume throughout the paper that
\begin{align} \label{eq:cond-K_as}\tag{K$_{as}$}
A(K):=\sup_{x\in \R^N} \int_{\{K_s(x,y)\neq 0\}} \frac{K_a^2(x,y)}{K_s(x,y)} \ dy<\infty.
\end{align}

\begin{remark}
\label{remark-condition-AS}
  (i) Condition (\ref{eq:cond-K_as}) is automatically satisfied if $K$ is symmetric, i.e., $K(x,y)=K(y,x)$ for all $x,y \in \R^N$, $x \not = y$. Moreover, we stress that all main results of the paper are already new in the case of a symmetric kernel function $K$.\\ 
(ii) We also remark that, if \eqref{eq:cond-K} and \eqref{eq:cond-A1} hold, then (\ref{eq:cond-K_as}) is already satisfied if 
$$
\sup_{x\in \R^N} \int_{\{K_s(x,y)\neq 0\}\cap B_\eps(x)} \frac{K_a^2(x,y)}{K_s(x,y)} \ dy<\infty \qquad \text{for some $\eps>0$,}
$$
since, for every $x \in \R^N$, the remaining part of the integral is bounded by 
$$
\|K_s(x,\cdot)\|_{L^1(\R^N \setminus B_{\eps}(x))} \le \Lambda \|j\|_{L^1(\R^N \setminus B_{\eps}(0))}<\infty
$$
Here, we used the obvious fact $K_a(x,y) \le K_s(x,y)$ for every $x,y \in \R^N$.\\ 
(iii) Condition (\ref{eq:cond-K_as}) is also satisfied automatically if $j \in L^1(\R^N)$, since  
$$
A(K) \le \sup_{x \in \R^N}\|K_s(x,\cdot)\|_{L^1(\R^N)}\le \Lambda \|j\|_{L^1(\R^N)}<  \infty
$$
in this case.
\end{remark}

\begin{definition}
  \label{def-class-solutions-intro}
(i) We let $V(\Omega|\R^N)$ denote the space of all functions $v \in L^2_{\loc}(\R^N)$ for which the seminorm $[v]_{V(\Omega|\R^N)}$ given by 
\begin{align*}
[v]^2_{V(\Omega|\R^N)} = \int \!\!\!\!\!\!\!\!\! \int_{\R^{2N}\setminus [\Omega^c]^2} \big(v(x) - v(y)\big)^2 j(x-y)dx dy 
\end{align*}
is finite.\\
(ii) We call $u \in V(\Omega|\R^N)$ a \textit{weak solution} of (\ref{general-op-eq}) in $\Omega$ if
\begin{equation}
  \label{eq:def-weak-sol-intro}
\cE(u,\phi)=
  \int_{\Omega}f \phi\,dx  \qquad \text{for all $\phi \in C^\infty_c(\Omega)$}
\end{equation}
with the associated bilinear form
\begin{equation}
  \label{eq:def-E-intro}
(v,w) \mapsto \cE(v,w):= \iint_{\R^{N}\R^{N}} \big(v(x)-v(y)\big) w(x) K(x,y) dy dx.
\end{equation}
\end{definition}
Here and in the following, we put $\Omega^c:= \R^N \setminus \Omega$, and we note that, similarly as in \cite[Proof of Lemma 2.4]{FKV15} the expression $\cE(u,\phi)$ is easily seen to be well-defined by (\ref{eq:def-E-intro}) in Lebesgue sense for $u \in V(\Omega|\R^N)$ and $\phi \in C^\infty_c(\Omega)$. We remark that the space $V(\Omega|\R^N)$ can be seen as a nonlocal counterpart of the space $H^1(\Omega)$. Moreover, we note that in the case of symmetric $K$ we can write 
$$ 
\cE(u,\phi):= \frac{1}{2} \iint_{\R^{N}\R^{N}} \big(u(x)-u(y)\big)\big(\phi(x)-\phi(y)\big)K(x,y) dy dx \quad \text{for $u \in V(\Omega |\R^N)$, $\phi \in C^\infty_c(\Omega)$.}
$$
Our first main result concerns the boundedness of weak solutions of (\ref{general-op-eq}).

\begin{thm}
  \label{new-boundedness-unified-version}
  Assume that \cref{eq:cond-K} holds with some measurable function $j: \R^N \setminus \{0\} \to [0,\infty)$ satisfying (\ref{eq:cond-A1}) for some $\gamma \le 2$, and suppose also that (\ref{eq:cond-K_as}) holds (if $j \not \in L^1(\R^N)$). Moreover, let $\Omega \subset \R^N$ be a bounded open set, and let $f,W \in L^\infty(\Omega)$. If
  \begin{equation}
    \label{eq:key-comparison-W-j}
  \Lambda \|W^+\|_{L^\infty(\Omega)} < \int_{\R^N}j(z)\,dz \le \infty  
  \end{equation}
  then every weak solution $u \in V(\Omega|\R^N)$ of (\ref{general-op-eq}) with $u \in L^\infty(\Omega^c)$
satisfies $u \in L^\infty(\Omega)$, and there exists a constant $C>0$ not depending on $f$ and $u$ with
  \begin{equation}
  \label{new-boundedness-unified-version-estimate}
\|u\|_{L^\infty(\Omega)} \le C\Bigl( \|f\|_{L^\infty(\Omega)} + \|u\|_{L^\infty(\Omega^c)} +\|u\|_{L^2(\Omega)} \Bigr).    
  \end{equation}
If, in addition, $W \le 0$ and the kernel $K$ is symmetric, the term $\|u\|_{L^2(\Omega)}$ can be dropped in (\ref{new-boundedness-unified-version-estimate}). 
\end{thm}

\begin{remark}
\label{remark-new-boundedness-unified-version}
\begin{itemize}[itemsep=0pt, topsep=2pt]
\item[(i)] Let us point out two special cases where condition \eqref{eq:key-comparison-W-j} holds. First, \eqref{eq:key-comparison-W-j} holds for a nonpositive $W \in L^\infty(\Omega)$ and an arbitrary nontrivial nonnegative function $j \in L^1(\R^N)$, and both \eqref{eq:cond-A1} and \eqref{eq:cond-K_as} are satisfied automatically in this case.
Second, condition \eqref{eq:key-comparison-W-j} holds for arbitrary $W \in L^\infty(\Omega)$ if $j \not \in L^1(\R^N)$, i.e., if
\begin{equation}
\int_{\R^N}j(z)\ dz =\infty. \tag{A2} \label{eq:cond-A2}
\end{equation}
\item[(ii)] It is clear that the term $\|u\|_{L^2(\Omega)}$ cannot be dropped in \eqref{new-boundedness-unified-version-estimate} in the case of general $W \in L^\infty(\Omega)$. Indeed, without the term $\|u\|_{L^2(\Omega)}$, the estimate \eqref{new-boundedness-unified-version-estimate} cannot hold for (the unbounded set of) eigenfunctions $u$ corresponding to some Dirichlet eigenvalue $\lambda>0$ of $\op$, which solve \eqref{general-op-eq} with $W \equiv \lambda$, $f \equiv 0$ and satisfy $u \equiv 0$ in $\Omega^c$.\\
\item[(iii)] The assumption $u \in L^\infty(\Omega^c)$ in \cref{new-boundedness-unified-version} can be weakened. More precisely, it is sufficient to assume that $u \in L^\infty(\Omega' \setminus \Omega)$ for some bounded open set $\Omega' \subset \R^N$ with $\Omega \Subset \Omega'$, and that
  \begin{equation}
    \label{eq:finite-j-tail}
T(u,\Omega,\Omega'):=  \sup_{x \in \Omega}\int_{\R^N \setminus \Omega'}|u(y)|j(y-x)\,dy < \infty.
  \end{equation}
  In this case, \eqref{new-boundedness-unified-version-estimate} needs to be replaced by
\begin{equation}
  \label{new-boundedness-main-theorem-estimate-general}
\|u\|_{L^\infty(\Omega)} \le C\Bigl( \|f\|_{L^\infty(\Omega)} + \|u\|_{L^\infty(\Omega' \setminus \Omega)} + \|u\|_{L^2(\Omega)} +T(u,\Omega,\Omega')\Bigr)    
  \end{equation}
  with $C>0$ depending on $\Omega'$. Again, the term $\|u\|_{L^2(\Omega)}$ can be dropped if the kernel $K$ is symmetric and $W$ is nonnegative. This generalization can be derived a posteriori from \cref{new-boundedness-unified-version}, see \cref{remark-new-section-boundedness} below.
	\end{itemize}
\end{remark}

Next we turn to uniform regularity estimates for solutions of \eqref{general-op-eq}. In this context, condition \eqref{eq:cond-A2} is a necessary minimal growth condition, since in the case $j \in L^1(\R^N)$ the operator $\op$ reduces to a sum of a multiplication and a convolution type integral operator, and therefore no gain of regularity is expected when passing from a bounded RHS $f$ in \eqref{general-op-eq} to an associated weak solution $u$.

Our second main theorem of this paper is concerned with the continuity of solutions of \eqref{general-op-eq}. For this, we need to assume, in addition, one of two alternative assumptions, which cover different classes of kernel functions $j$. Both assumptions are formulated for a given radius $R_0 > 0$. The first assumption is a locally scaling-invariant doubling property of $j$ on annular domains.
\begin{equation}\label{eq:cond-A-3-1} \tag{A$3_1$} 
\begin{split}
&\text{There exist } c_0 \in (0,1) \text{ and } \sigma \in (0,1) \text{ such that} \\  
& c_0 \leq \frac{j(z)}{j(y)} \leq c_0^{-1}  \quad \text{for every $z,y \in B_{R_0} \setminus \{0\}$ with  $1-\sigma \le \tfrac{|z|}{|y|} \le 1+\sigma$.}
\end{split} 
\end{equation}
The following alternative condition is very different in nature and requires $j$ to be of locally bounded variation. 
\begin{equation}\label{eq:cond-A-3-2} \tag{A$3_2$}
\begin{split}
&\text{There exists $M>0$ such that } j \in BV_{\loc}(B_{R_0} \setminus \{0\}) \text{ and } \\
&\|Dj\|(B_{R} \setminus \overline{B_r}) \le \frac{M}{r} \int_{B_r^c} j(z)\,dz \quad \text{for $0<r<R<R_0$.} 
\end{split}  
\end{equation}
Note that \eqref{eq:cond-A-3-1} is a quite natural assumption which, in particular, implies positivity of $j$ in $B_{R_0} \setminus \{0\}$. On the other hand, condition \eqref{eq:cond-A-3-2} seems to be completely new. In particular, it allows $j$ to vanish in sets arbitrarily close to zero. We refer to  the comments after \cref{thm:continuity} below and to \cref{sec:examples} for several examples satisfying the aforementioned conditions.

\begin{thm}
  \label{thm:continuity}
  Suppose that \cref{eq:cond-K}, \cref{eq:cond-K_as}, \cref{eq:cond-A1} with $\gamma\in(0,1)$, and \cref{eq:cond-A2} hold for a nonnegative function $j \in L^1_{loc}(\R^N \setminus \{0\})$ and assume further that \eqref{eq:cond-A-3-1} \underline{or} \eqref{eq:cond-A-3-2} is satisfied for some $R_0>0$. 
  Moreover, let $\Omega \subset \R^N$ be open and $f, W \in L^\infty_{\loc}(\Omega)$. Then every weak solution $u \in V(\Omega|\R^N) \cap L^{\infty}(\R^N)$ of (\ref{general-op-eq}) has a continuous representative.

  Moreover, for open subsets $A \Subset B \Subset \Omega$, there exists a nondecreasing continuous function $\omega : [0, \infty) \to [0, \infty)$ independent of $f$ and $u$ with $\omega(0)=0$ and the property that 
\begin{align}
  \label{new-continuity-main-theorem-estimate}
		|u(x) - u(y)| \leq \omega(|x-y|)\Bigl(\|u\|_{L^{\infty}(\R^N)}+ \|f\|_{L^\infty(B)}\Bigr) \qquad \text{for all $x,y \in A$.}
\end{align} 
\end{thm}

\begin{remark}
  \label{new-boundedness-main-theorem-remark2}
  \begin{itemize}[itemsep=0pt, topsep=2pt]
  \item[(i)] Instead of assuming $u \in L^\infty(\R^N)$, it is sufficient to assume in \cref{thm:continuity} that $u \in L^\infty_{loc}(\Omega)$ and that the function
    $$
    x \mapsto \int_{\R^N \setminus B_r(x)}|u(y)|j(y-x)\,dy 
    $$
    is locally bounded in $\Omega$ for every $r>0$. Then \eqref{new-continuity-main-theorem-estimate} needs to be replaced by 
\begin{align}
  \label{new-continuity-main-theorem-estimate-1}
		|u(x) - u(y)| \leq \omega(|x-y|) \Bigl(\|u\|_{L^{\infty}(B)}+ \|f\|_{L^\infty(B)}+T(u,A,B)\Bigr) 
\end{align}
for all $x,y \in A$, where $T(u,A,B)$ is defined as in \eqref{eq:finite-j-tail}, see \cref{bounded j tail defi} below. In view of Theorem~\ref{new-boundedness-unified-version} (and \cref{remark-new-boundedness-unified-version}(iii)), we can also drop the boundedness assumption on $u$ in $\Omega$ completely if we assume additionally that $f, W \in L^\infty(\Omega)$. More precisely, we may replace assumption $u \in L^\infty(\R^N)$ by $u \in L^\infty(\Omega^c)$ (or $u\in L^{\infty}(\Omega'\setminus \Omega)$ for some $\Omega'\Subset \R^N$ with $\Omega\Subset \Omega'$) in this case.
\item[(ii)] By a simple localization argument, the conclusions of \cref{thm:continuity} also hold if the assumption $u \in V(\Omega|\R^N)$ is replaced by assuming only that $u \in V(\Omega'|\R^N)$ for all open subsets $\Omega' \Subset \Omega$.
\end{itemize}
\end{remark}

As remarked earlier, \cref{thm:continuity} is formulated with the aim to assume as little as possible on the kernel $K(x,y)$ resp. the density $j$ in addition to \cref{eq:cond-A1}, \cref{eq:cond-A2}, and \cref{eq:cond-K_as}. In particular, we wish to point out the following:

\begin{enumerate}[itemsep=0pt, topsep=2pt]
\item[(i)] Our assumptions are weaker than in existing regularity results for weakly singular nonlocal operators. In particular, this applies to \cite{KaMi17} and the references therein.  
\item[(ii)] We are able to treat operators with anisotropic kernels. The new condition \eqref{eq:cond-A-3-2} allows a remarkable degree of freedom in this regard, see Section~\ref{sec:examples} below for a detailed discussion and examples. 
\item[(iii)] \cref{thm:continuity} applies to weak (or variational) solutions, and its proof is based on combining ideas of \cite{Sil06} and \cite{KaMi17} with an energy approach based on nonlocal bilinear forms.
\end{enumerate}
The new features mentioned in (i) and (ii) are visible in particular for density functions of the form
\begin{equation}
  \label{eq:j-ex-intro}
z \mapsto j(z) = \ell(|z|)|z|^{-N},
\end{equation}
where $\ell \in L^\infty_{\loc}(0,\infty)$ is a nonnegative function satisfying for some $\gamma\in(0,1)$
		\begin{align}\label{eq:assumption-l1}
		\int_0^{\infty}\min \{1,s^{\gamma}\}\frac{\ell(s)}{s}\ ds <\infty\quad \text{and}\quad \int_0^{1}\frac{\ell(s)}{s}\ ds=\infty.
		\end{align}
		It follows readily from (\ref{eq:assumption-l1}) that $j$ satisfies \eqref{eq:cond-A1}, \eqref{eq:cond-A2}. 
Moreover, \eqref{eq:cond-A-3-1} is ensured by the simple comparability condition 			\begin{align}
			\label{eq:local-radial-positivity}
			c_1  \le \frac{\ell(r)}{\ell(R)} \le c_2 \quad \text{for $r,R\in(0,R_0)$ with $1-\sigma<\frac{r}{R}< 1+\sigma$,}
			\end{align}
                        for some $R_0>0$. This condition includes regularly varying functions at zero, which were considered in \cite{KaMi17} in this context,  but it is much more general. In particular, it is satisfied if $0 < \inf_{(0,R_0)} \ell \le \sup_{(0,R_0)} \ell < \infty$ for some $R_0 >0$.

                        In Section~\ref{sec:examples} below, we will also consider $j$ given by (\ref{eq:j-ex-intro}) with a nonnegative function $\ell$ on $(0,\infty)$ with infinitely many zeros close to zero. In this case \eqref{eq:cond-A-3-1} cannot be satisfied, but there are examples in this class satisfying \eqref{eq:cond-A-3-2}. However, the main motivation for introducing \eqref{eq:cond-A-3-2} is given by (nonradial) anisotropic density functions $j$. A typical example is given by the class of density functions
\begin{equation}
  \label{eq:j-ex-intro-1}
z \mapsto j(z) = 1_A(\frac{z}{|z|})|z|^{-N}
\end{equation}
where $A$ is a subdomain of the unit sphere $S^{N-1}$ of class $C^1$. Note that $j(z)$ then vanishes outside of the cone $C_A:= \{r \theta\::\: r >0,\theta \in A\}$. This example and the more general class of density functions $z \mapsto j(z) = 1_A(\frac{z}{|z|})\ell(|z|)|z|^{-N}$ are discussed in detail in Section~\ref{sec:examples} below.

Before we discuss previous results in the direction of \cref{thm:continuity}, let us comment on different notions of symmetry related to our setting, i.e., with regard to the functions $K(x,y)$ and $j(h)$ from \eqref{eq:cond-K}. The operator $\op$ as in \eqref{def:op-new} is a symmetric operator in $L^2(\R^N)$ if $K(x,y)=K(y,x)$ for all $x \ne y$. In the translation-invariant case $K(x,y)=j(y-x)$ this condition reduces to $j(h) = j(-h)$ for all $h \ne 0$. Our approach works without these assumptions. We only require that the asymmetric part $K_a$ to be dominated by the symmetric part $K_s$ in a certain sense, see the examples in \cref{sec:examples}.  

There is another kind of symmetry that is helpful when studying nonlocal operators of the form of \eqref{def:op-new}. The condition $K(x,x+h)=K(x,x-h)$ for all $x$ and $h \ne 0$ is typically employed when working with techniques stemming from the theory of second-order differential operators in non-divergence form. As before, in the translation-invariant case this condition reduces to $j(h) = j(-h)$ for all $h \ne 0$. If we assume \eqref{eq:cond-A1} with $\gamma> 1$, then $-\op \phi$ is bounded for a smooth bounded function $\phi$ under the additional assumption (\!\cite{JaWe19})
\begin{equation}\label{above-higher}
	\sup_{x\in \R^N}\int_{\R^N}\min(1, |h|)\big|K(x,x+h)-K(x,x-h)\big|\ dh < \infty \,.
\end{equation}
This extra condition would allow us to extend the proof of \cref{thm:continuity} to the case $\gamma > 1$ using only slight modifications. 

Let us comment on some results related to \cref{thm:continuity}. 

In the case $K(x,x+h)=K(x,x-h)$ with $j(h) = |h|^{-d-2s}$ for $0 < s < 1$, \cref{thm:continuity} has been established first in \cite{BaLe02} using methods from Potential Theory and Stochastic Analysis. The authors makes use of the fact that $\cL$ is the generator of a Markov jump process. This approach has been used in order to cover more general, also anisotropic, cases, see \cite{SoVo04, BaKa05a, BaKa05b, KaMi13}. The article \cite{Sil06} introduces PDE techniques for results like \cref{thm:continuity} allowing for a broad range of linear and nonlinear integro-differential equations. We make use of the method developed in \cite{Sil06}. \cite{KaMi17} has extended the techniques from \cite{BaLe02, Sil06} to operators with weakly singular kernels $j$ including $j(h) = 1_{B_1} (h) |h|^{-d}$. In the current work, our assumptions are weaker than those in \cite{KaMi17}. The approaches of \cite{BaLe02, Sil06, KaMi17} are not robust w.r.t. the order of differentiability, i.e., the regularity estimate degenerate for $s \to 1$. Robust techniques using localized ABP-estimates have been developed for linear and nonlinear versions of the operator $-\op$ in \cite{CaSi09} and extend in \cite{ChDa12} to the non-symmetric case. See \cite{ScSi16} for references to works that employ ideas stemming from \cite{CaSi09} and that allow for non-symmetric cases, i.e., not assuming $K(x,x+h)=K(x,x-h)$ for all $h \ne 0$. 

Variational calculus can be applied successfully to operators $\op$ as in \eqref{def:op-new}. In this context, symmetry is given when assuming $K(x,y)=K(y,x)$ for all $x \ne y$. See \cite{FKV15, FoKa22} for the general setting. Questions of well-posedness in the case of jump kernels with very weak singularities can be found in \cite{CodP18}. In this short review, let us concentrate on Hölder regularity results for nonlocal operators with merely measurable kernels $K(x,y)$. 

The De Giorgi-Nash-Moser theory has been extended in \cite{Kas09} to equations of the form (\ref{general-op-eq}) with $\op$ as in \eqref{def:op-new} under the assumption $K(x,y)=K(y,x)$ with $j(h) = |h|^{-d-2s}$ for $0 < s < 1$.  \cite{DyKa20}, \cite{DKP14}, \cite{DKP16}, and \cite{Coz17} have proved H\"older regularity estimates and Harnack inequalities for weak solutions by adopting the Moser iteration and the De Giorgi iteration to the nonlocal setting. There are numerous extensions to time-dependent problems, operators with singular jumping measures or to equations on metric measure spaces that we do not mention here. The general non-symmetric case has recently been studied in \cite{KaWe22a}. Under conditions similar to \eqref{eq:cond-K_as} and $j(h) = |h|^{-d-2s}$, the authors provide Hölder-regularity of weak solutions. 

Finally, let us mention the two articles \cite{CS22, DRSV22}, which do not focus on regularity in the case of bounded measurable coefficients or jumping kernels. However, both works contain interesting results in the non-symmetric case. \cite{CS22} focuses on existence of classical solutions for operators with kernels of the form $K(x,y)=\frac{k(x-y)}{|x-y|^N}$. The authors assume $k$ to be a differentiable function on $\R^N\setminus\{0\}$, possibly non-symmetric, satisfying $\Lambda^{-1}\leq k\leq \Lambda$ for some $\Lambda\geq 1$. \cite{DRSV22} establishes interior and global regularity in the translation-invariant non-symmetric case, i.e., under the assumption $K(x,x+h)=j(h)$ for some function $j \in L^1_{\loc}(\R^N\setminus\{0\})$, which is positively homogeneous of degree $-N-2s$ and does not have to be even. The authors consider rather general 
nonsymmetric stable operators and prove interior and boundary regularity in H\"older spaces depending on the exponent $s\in (0,1)$.

The article is organized as follows. In \cref{sec:preliminaries} we discuss the assumptions and provide examples of densities $j$. In this section we also provide the definitions of the function spaces we use and collect general properties of these spaces. In \cref{sec:boundedness} we present a boundedness statement, \cref{new-boundedness-unified-version-section}, which immediately implies \cref{new-boundedness-unified-version} and \cref{sec:growth-lemma} contains the growth lemma, \cref{thm:growth-lemma}, which is our main auxiliary result. \cref{thm:continuity} is a direct consequence of \cref{thm:continuity:section}, which we prove in \cref{sec:continuity}.

\paragraph{Acknowledgment.} Moritz Kassmann gratefully acknowledges financial support by the German Research Foundation (SFB 1283 - 317210226). 

\section{Preliminaries and Examples}\label{sec:preliminaries}

We write $B_r$ instead of $B_r(0)=\{x \ \in \R^N \,: \, 0 \leq |x| < r\}$ and $S_r:= \partial B_r$ for $r>0$.

\subsection{Examples}\label{sec:examples}

Let us discuss the conditions \eqref{eq:cond-A-3-1} and \eqref{eq:cond-A-3-2}. As already mentioned in the introduction, condition \eqref{eq:cond-A-3-1} is weaker than the ones assumed in \cite{KaMi17} (see conditions ($\ell_1$)--($\ell_3$) there), but it does not allow for kernels $j$ which vanish on subsets arbitrarily close to the origin. In order to deal with kernels of this type, condition \eqref{eq:cond-A-3-2} can be used, see \cref{examples:radial} below. We also note that none of the two conditions, \eqref{eq:cond-A-3-1} and \eqref{eq:cond-K_as}, implies the other one, even not in the translation-invariant case in one space dimension. 
The kernel $K(x,y)=j(y-x)$ with 
\begin{align*}
j(h) = h^{-1-1/2} \text{ for } h>0 \text{ and } j(h) = 2 (-h)^{-1-1/2} \text{ for } h<0 
\end{align*} 
satisfies \eqref{eq:cond-A-3-1} but not \eqref{eq:cond-K_as}. The choice
\begin{align*}
	j(h) = |h|^{-1-\frac{1}{6}(\cos(h^{-1})+4)}+h|h|^{-2-\frac16}1_{\{|h|\leq1\}}(h)
\end{align*}
provides an example satisfying \eqref{eq:cond-K_as} but not \eqref{eq:cond-A-3-1}. Both examples satisfy \eqref{eq:cond-A1}, \cref{eq:cond-A2}, and \eqref{eq:cond-A-3-2}.

Regarding  \eqref{eq:cond-A-3-2}, let us recall that, if $\Omega \subset \R^N$, $BV_{\loc}(\Omega)$ is defined as the space of all functions $g \in L^1_{\loc}(\Omega)$ such that the total variation 
$$
\|Dg\|(\Omega') = \sup \Bigl\{ \int_{\Omega'}\,g \,\div\, w\,dx \::\: w \in C^1_c(\Omega',\R^N),\,\|w\|_{L^\infty(\Omega')} \le 1 \Bigr\}
$$
is finite for all open subsets $\Omega' \Subset \Omega$. The relevance of the space $BV_{\loc}(\Omega)$ in the present context arises from the estimate 
\begin{align}
\label{eq:bv-loc-key-property}
\int_{K}|g(z+x)-g(z)|\,dz \le \eps \|Dg\|(\Omega) 
\end{align}
which holds for $g \in BV_{\loc}(\Omega)$, Borel subsets $K \Subset \Omega$, $0<\eps < \dist(K,\partial \Omega)$, and $x \in B_\eps$, see e.g. \cite[Remark 3.25, p.134]{AFP00}. 
The estimate \eqref{eq:bv-loc-key-property} will be used at a key point in the present paper. 

Next, let us provide several examples. 

\begin{example}\label{examples:radial} {\it (Radial kernels)} We first consider kernels of the form 
		\begin{align}
		\label{eq:kernel-radial}
		z \mapsto j(z) = \ell(|z|)|z|^{-N},
		\end{align}
		where $\ell \in L^\infty_{\loc}(0,\infty)$ is a nonnegative function satisfying (\ref{eq:assumption-l1}) for some $\gamma\in(0,1)$. As mentioned already in the introduction, it then follows readily that $j$ satisfies \eqref{eq:cond-A1}, \eqref{eq:cond-A2}.

                Moreover, condition \eqref{eq:cond-A-3-1} is ensured by assuming (\ref{eq:local-radial-positivity}) for some $R_0>0$.
                
As noted before, \eqref{eq:cond-A-3-1} requires, in particular, local positivity of the kernel close to zero. To give an example of an admissible radial kernel of the form \eqref{eq:kernel-radial} with zeros arbitrarily close to $0$, we need to consider condition \eqref{eq:cond-A-3-2} instead. It is easy to see that \eqref{eq:cond-A-3-2} is satisfied if, for some constants $R_0, \widetilde{M}>0$, the function $\ell$ is locally absolutely continuous on $(0,R_0]$ and satisfies
			\begin{align}
			\label{eq:A4-2-radial}
			\int_r^{R_0} \frac{|\ell'(\tau)|}{\tau}\,d\tau \le \frac{\widetilde{M}}{r}\int_r^\infty \frac{|\ell(\tau)|}{\tau}d\tau \qquad \text{for $0<r< R_0$.}
			\end{align}
			An example of a function $\ell$ satisfying \eqref{eq:A4-2-radial} for $R_0>0$ sufficiently small and $\widetilde{M}>0$ sufficiently large is given by 
			$$
			r \mapsto \ell(r)=  \left \{
			\begin{aligned}
			&1-\sin (\ln r), &&\qquad r \le 1\\ 
			&0, &&\qquad r > 1\\ 
			\end{aligned}
			\right.
			$$
			
			Moreover, the kernel $j$ defined by \eqref{eq:kernel-radial} also satisfies the conditions \eqref{eq:cond-A1}, \eqref{eq:cond-A2} in this case.
\end{example}

Next we provide examples with non-radial functions $j$.

\begin{example}\label{examples:nonradial} {\it (Non-radial kernels)} We consider kernels $j$ of the form 
		\begin{align}
		\label{eq:kernel-nonradial}
		z \mapsto j(z) = \ell(|z|)\lambda(\frac{z}{|z|})|z|^{-N},
		\end{align}
		where $\ell \in L^\infty_{\loc}(0,\infty)$ and $\lambda \in L^{\infty}(S^{N-1})$ are nonnegative functions.
		\begin{itemize}
                \item[1.] If $\inf \limits_{S^{N-1}}\lambda >0$, and if $\ell$ satisfies \eqref{eq:assumption-l1} for some $\gamma \in (0,1)$ and \eqref{eq:local-radial-positivity} for some $R_0,c_1,c_2>0$, then $j$ satisfies \eqref{eq:cond-A1}, \eqref{eq:cond-A2}, and \eqref{eq:cond-A-3-1}. However, in this case we can replace $j$ also by a radial function satisfying these conditions and
                  (\ref{eq:cond-K}).
			\item[2.] Using condition \eqref{eq:cond-A-3-2} in place of \eqref{eq:cond-A-3-1}, we may also consider highly anisotropic kernels of the form \eqref{eq:kernel-nonradial} where $\lambda$ may have a large zero set. Specifically, let us consider the case $\lambda= 1_A$, where $A \subset S^{N-1}$ is a subdomain of class $C^1$ with $A = -A$ and $\cH_{N-1}(A)>0$. Moreover, suppose that $\ell:(0,\infty) \to \R$ is a nonnegative function satisfying \eqref{eq:assumption-l1} for some $\gamma \in (0,1)$ and which, for some constants $R_0, \widetilde{M}>0$, is locally absolutely continuous on $(0,R_0]$ and satisfies \eqref{eq:A4-2-radial}. We then consider the kernel $z \mapsto j(z)=|z|^{-N}\ell(|z|)1_A(\frac{z}{|z|})$ which vanishes outside of the cone
			$$
			C_A := \{z \in \R^N \setminus \{0\}\::\: \frac{z}{|z|} \in A\}.
			$$ 
			Then $j$ satisfies \eqref{eq:cond-A1}, \eqref{eq:cond-A2} and we claim that it also satisfies \eqref{eq:cond-A-3-2} in this case. In order to estimate the total variation $\|Dj\|(B_{R} \setminus \overline B_r)$ for $0<r<R<R_0$, we consider $w \in C^1_c(B_{R} \setminus \overline B_r, \R^N)$ with $\|w\|_{L^\infty} \le 1$, and we consider the sets 
			$$
			K_r^R:= [B_{R} \setminus \overline B_r] \cap C_A,\qquad  \Gamma_r^R:= \{z \in B_{R} \setminus \overline B_r \::\: \frac{z}{|z|} \in \partial A\} \subset \partial K_r^R.
			$$
			Since $w$ vanishes on $\partial K_r^R \setminus \Gamma_r^R \subset S_r \cup S_R$, we then have with the exterior normal field $\eta$ of $K_r^R$ that 
			$$
			\int_{B_{R} \setminus B_r} j \,\div\, w\,dx=\int_{K_r^R} j \,\div\, w\,dx= \int_{\partial K_r^R}j w d \sigma - \int_{K_r^R} (\nabla j \cdot w)dx = \int_{\Gamma_r^R}j w\eta d \sigma - \int_{K_r^R} (\nabla j \cdot w)dx,
			$$
			and therefore 
			\begin{align}
			\label{est-var-measure}
			\|Dj\|(B_{R} \setminus \overline B_r) \le  \int_{\Gamma_r^R} j d\sigma  + \int_{K_r^R} |\nabla j|\,dx.
			\end{align}
			Moreover,   
			\begin{align}
			\int_{\Gamma_r^R} j d\sigma &= \int_{r}^{R} \tau^{-N} \cH_{N-2}(\Gamma_r^R \cap S_\tau)\ell(\tau)\,d\tau = \cH_{N-2}(\partial A) \int_{r}^{R} \tau^{-2}\ell(\tau)\,d\tau \nonumber\\
			&\le \frac{\cH_{N-2}(\partial A)}{r} \int_{r}^{R} \frac{\ell(\tau)}{\tau}\,d\tau \le \frac{\cH_{N-2}(\partial A)}{r\cH_{N-1}(A)} \int_{B_r^c} j(z)\,dz\label{gamma-star-part},
			\end{align}
			and by \eqref{eq:A4-2-radial} we have 
			\begin{align}
			\int_{K_r^R} |\nabla j|\,dx &= \cH_{N-1}(A)\int_{r}^R \tau^{N-1}\bigl|\tau^{-N}l'(\tau)- N \tau^{-N-1}l(\tau)\bigr|\,d\tau \nonumber\\
			&\le \cH_{N-1}(A) \int_{r}^{R_0} \Bigl(\frac{|\ell'(\tau)|}{\tau}+ N \frac{\ell(\tau)}{\tau^2}\Bigr)\,d\tau  \le \frac{\cH_{N-1}(A)(\widetilde{M} +N)}{r}\int_{r}^R\frac{\ell(\tau)}{\tau}\,d\tau \nonumber\\
			&\le \frac{\widetilde{M} +N}{r} \int_{B_r^c} j(z)\,dz \qquad \text{for $0<r<R <R_0$.} 
			\label{K-r-est}
			\end{align}
			Combining \eqref{est-var-measure}, \eqref{gamma-star-part}, and \eqref{K-r-est}, we conclude
			$$
			\|Dj\|(B_{R} \setminus \overline B_r) \le \frac{M}{r} \int_{B_r^c} j(z)\,dz
			\quad \text{ for } 0<r<R<R_0 
			$$
			with $M:= \Bigl(\frac{\cH_{N-2}(\partial A)}{\cH_{N-1}(A)}+ 3\widetilde{M} + N\Bigr)$.
	\end{itemize}
\end{example}

\begin{remark}
We close the discussion on examples with a remark concerning assumption \cref{eq:cond-K_as} in \cref{thm:continuity}. Indeed, in our proofs we need the existence of some constant $c\in \R$ (not necessarily positive) such that
\begin{equation}\label{necessary-inequality}
\cE(u,u)=\iint_{\R^N\,\R^N}(u(x)-u(y))u(x)K(x,y)\ dydx\geq c\|u\|_{L^2(\R^N)}^2,
\end{equation}
Indeed, this is still true with a somewhat more general assumption in place of \cref{eq:cond-K_as}. In the following, we say that $(\tilde{K}_{as})$ is satisfied, if there is a symmetric kernel $\tilde{k}:\R^N\times \R^N\to[0,\infty]$ with $|\{y\in \R^N\;:\; \tilde{k}(x,y)=0,\,K_a(x,y)\neq0\}|=0$ for all $x\in \R^N$, 
\begin{subequations}
\makeatletter
        \def\@currentlabel{\~K$_{as}$}
				\makeatother
 \label{eq:cond-K_as2}
\renewcommand{\theequation}{\~K$_{as}$.\arabic{equation}}
\begin{align}
\iint_{\R^N\,\R^N}(u(x)-u(y))^2\tilde{k}(x,y)\ dydx&\leq \iint_{\R^N\,\R^N}(u(x)-u(y))^2K_s(x,y)\ dydx\label{eq:cond-K_as2a} \\
 &\qquad\qquad\text{for all $u\in V(\R^N|\R^N)$, and}\notag\\
\tilde{A}(K, \tilde{k}):=\sup_{x\in \R^N} \int_{\{\tilde{k}(x,y)\neq 0\}} &\frac{(K_a(x,y))^2}{\tilde{k}(x,y)} \ dy<\infty.\label{eq:cond-K_as2b} 
\end{align}
\end{subequations}
If \eqref{eq:cond-K_as2} is satisfied in place of \eqref{eq:cond-K_as}, then \eqref{necessary-inequality} still holds, see \cref{lemma:garding}, and thus the conclusion of \cref{new-boundedness-unified-version} and \cref{thm:continuity} continue to hold.\\
Clearly, \eqref{eq:cond-K_as} implies \eqref{eq:cond-K_as2} by choosing $\tilde{k}=K_s$. The opposite, however, is not true: In \cite[Example (12), Section 6]{FKV15}, a kernel is discussed of the form
$$
k(x,y)=|x-y|^{-N-\alpha}1_{C_1}(x-y)+|x-y|^{-N-\beta}1_{C_2}(x-y)1_{B_1}(x-y),
$$
where $0<\beta<\frac{\alpha}{2}$ and $C_j=\{h\in \R^N\setminus\{0\}\;:\; h/|h|\in I_j\}$, $j=1,2$. Here, $I_1,I_2$ are nonempty disjoint subdomains of $\mathbb{S}^{N-1}$ satisfying $I_1=-I_1$ and $|(-I_2)\setminus I_2|>0$. If $\alpha<1$, $I_1$ and $I_2$ are of class $C^1$ in $\mathbb{S}^{N-1}$, then in view of \cref{examples:nonradial} it follows that this kernel satisfies \eqref{eq:cond-K}, \eqref{eq:cond-A1}, \eqref{eq:cond-A2}, \eqref{eq:cond-A-3-2}, and \eqref{eq:cond-K_as2}. As shown in \cite{FKV15}, however, this kernel does not satisfy \eqref{eq:cond-K_as}. 
\end{remark}

\subsection{Function spaces}\label{sec:spaces}

Let $\Omega \subset \R^N$ be an open set. We assume conditions $\eqref{eq:cond-A1}, \eqref{eq:cond-A2}$ to hold true. Let us quickly recall the relevant function spaces (c.f. \cref{def-class-solutions-intro}). We refer to \cite{FKV15}, \cite{JaWe19}, \cite{FoKa22}, and \cite{FeJa22} for detailed discussions. 

\begin{definition}
The space $V(\Omega|\R^N)$ consists of all $v \in L^2_{\loc}(\R^N)$, for which the seminorm $[v]_{V(\Omega|\R^N)}$ given by 
\begin{align*}
[v]^2_{V(\Omega|\R^N)} = \int_\Omega \int_{\R^N} \big(v(x) - v(y)\big)^2 \big(j(y-x) + j(x-y)\big) dy dx 
\end{align*}
is finite. The space $V_\Omega(\R^N)$ consists of all $v \in V(\R^N|\R^N)$ with $v=0$ in $\R^N \setminus \Omega$. The space $V_{\loc}(\Omega|\R^N)$ consists of all $v \in L^2_{\loc}(\R^N)$, for which the seminorm $[v]_{V(U|\R^N)}$ is finite for every open set $U \Subset \Omega$.
\end{definition}	

\begin{remark}
Note that the finiteness of $[v]_{V(\Omega|\R^N)}$ implies a certain weighted integrability of $v$ in all of $\R^N$ and it is equivalent to 
\begin{align*}
	\iint_{(\Omega^c \times \Omega^c)^c} \big(v(x) - v(y)\big)^2 j(y-x) dy dx < \infty.
\end{align*}
The spaces $V_{\Omega}(\R^N)$, $V(\Omega|\R^N)$, and resp. $V_{\loc}(\Omega|\R^N)$ can be seen as a nonlocal variants of $H^1_0(\Omega)$, $H^1(\Omega)$, and resp. $H^1_{\loc}(\Omega)$. We note that if $\Omega$ is bounded, then $V_{\Omega}(\R^N)$ is a Hilbert space with the semi-norm defined above (which is a norm on $V_\Omega(\R^N)$ in this case).
\end{remark}

The proof of the following observations is straightforward. 
\begin{lem}\label{lem:local-cutoff}
  \begin{enumerate}
  \item[(i)] The space $V(\Omega|\R^N)$ contains all constant functions. Moreover, if $u \in V(\Omega|\R^N)$, then also $u^\pm \in V(\Omega|\R^N)$.
  \item[(ii)] If $u \in V_{loc}(\Omega|\R^N)$, then also $u^\pm \in V_{loc}(\Omega|\R^N)$.
\item[(iii)]   Let $u\in V_{\loc}(\Omega|\R^N)$ and $U \Subset \Omega$ open. If $u\equiv 0$ in $\Omega\setminus U$, then $u 1_U\in V_U(\R^N) \subset V_\Omega(\R^N)$.
  \end{enumerate}
\end{lem}

Next, let us assume $K$ and $j$ satisfy \eqref{eq:cond-K}. If $K, u,v$ are sufficiently regular, then
\begin{align*}
\int_{\R^N}  \op u(x) v(x) dx = \iint_{\R^{N}\R^{N}} \big(u(x)-u(y)\big) v(x) K(x,y) dy dx =: \cE(u,v) \,.
\end{align*}

If $K$ is symmetric and $j$ is even, then
\begin{align*}
\cE(u,v) = \frac12 \iint_{\R^{N}\R^{N}} \big(u(x)-u(y)\big) \big(v(x)-v(y)\big) K(x,y) dy dx  \,.
\end{align*}

Note that the term $|\cE(u,v)|$ is well-defined and finite, if $u,v \in V(\R^N|\R^N)$ or, if $u \in V(\Omega|\R^N)$ and $v \in V_\Omega(\R^N)$. Moreover, we emphasize that $C^{\infty}_c(\Omega)$ is dense in $V_{\Omega}(\R^N)$, if $\Omega$ is an open set with continuous and compact boundary (see \cite[Theorem 3.76]{F20}). Thus, the following definition of solution is a generalization of the one stated in the introduction.

\begin{definition}\label{def:weak-sol}
Let $f \in L^2_{\loc}(\Omega)$ and $u \in V_{\loc}(\Omega|\R^N)$. We say that $u$  satisfies 
\begin{equation}\label{def:subsolution}
\op u \leq f \quad \text{ weakly in } \Omega
\end{equation}
if for every open $U \Subset \Omega$ and every nonnegative $\varphi \in V_{U}(\R^N)$ the inequality $\cE(u, \varphi) \leq (f, \varphi)$ holds. In this case, we also call $u$ a weak subsolution  (of \eqref{def:subsolution}). The respective definitions for a weak supersolution or a weak solution are analogous. We also use the term 
\textit{variational} sub- or supersolution.
\end{definition}

In our analysis, the G{\aa}rding inequality for nonlocal bilinear forms as stated in \cite{FKV15} is useful to deal with the case of $K$ being non-symmetric.

\begin{lem}[G{\aa}rding inequality, \cite{FKV15}]\label{lemma:garding}
	Let $K$ satisfy \eqref{eq:cond-K}, with $j$ satisfying \eqref{eq:cond-A1} for some $\gamma\in(0,2)$, and \eqref{eq:cond-K_as2}. Then there is $c>0$ depending only on $\tilde{A}(K,\tilde k)$ in \eqref{eq:cond-K_as2} with 
	\begin{equation}
		\label{eq:garding-inequality}
		\cE(u,u)\geq \frac{1}{4}\iint_{\R^{N}\R^{N}} \big(u(x)-u(y)\big)^2  K_s(x,y) dy dx -c \|u\|_{L^2(\R^N)}^2\quad\text{for all $u\in V(\R^N|\R^N)$.}
	\end{equation}
\end{lem}
 
The double integral in the lower bound of the G{\aa}rding inequality can be estimated further with the Poincar\'e inequality:

\begin{lem}[Poincar\'e inequality, Lemma 2.7 \cite{FKV15}]\label{lemma:poincare}
Let $K$ and $j$ satisfy \eqref{eq:cond-K} and assume that $j$ does not vanish identically, i.e.,
\begin{equation}\label{eq:poincare}
\text{the set $\{z \in \R^N\::\: j(z)>0\}$ has positive Lebesgue measure.}
\end{equation}
Then 
$$
\lambda_1(\Omega):=\inf_{\substack{u\in V_{\Omega}(\R^N)\\ \|u\|_{L^2(\Omega)}=1}}\quad \iint_{\R^N\,\R^N}(u(x)-u(y))^2K_s(x,y)\ dxdy>0
$$
for every open bounded set $\Omega\subset \R^N$.
\end{lem}

\begin{definition}\label{bounded j tail defi}
Let $u\in L^1_{loc}(\R^N)$. We say that $u$ has a bounded $j$-tail in $\Omega$, if
\begin{equation}\label{bounded j tail function defi}
x\mapsto \int_{\R^N\setminus B_r(x)}|u(y)|j(y-x)\,dy
\end{equation}	
belongs to $L^{\infty}_{loc}(\Omega)$ for every $r>0$. In this case, we set for $A\Subset B\Subset \Omega$, see also \eqref{eq:finite-j-tail},
	$$
	T(u,A,B):= \sup_{x \in A} \int_{\R^N \setminus B}|u(y)|j(x-y)\,dy.
	$$
\end{definition}

\begin{remark} Finally, let us comment on the Definition of $u$ having a bounded $j$-tail (see also \cref{new-boundedness-main-theorem-remark2}).
\begin{enumerate}
\item[(i)] If $j(h)=|h|^{-N-\alpha}$ for some $\alpha>0$, then every function $u\in L^1(\R^N;(1+|x|)^{-N-\alpha})$ has a bounded $j$-tail in $\R^N$.
\item[(ii)] If $u\in L^2_{\loc}(\R^N)$, and $j\in L^2_{\loc}(\R^N\setminus \{0\})$, then the local boundedness of the function in \eqref{bounded j tail function defi} for any $r>0$ is equivalent to the following:
$$
\text{There is $r>0$ such that}\quad \sup_{x\in A}\int_{\R^N\setminus B_r}|u(x+h)|j(h)\ dh<\infty\quad\text{for all $A\Subset \Omega$.}
$$
\item[(iii)] If $v\in L^p(\R^N)$ for some $p\in[1,\infty]$ and $j\in L^{p'}(\R^N\setminus B_r)$ for all $r>0$, where $\frac1p+\frac1{p'}=1$, then $u$ has a bounded $j$-tail in $\R^N$.
\end{enumerate}
\end{remark}

\section{Boundedness of solutions}\label{sec:boundedness}

This section is devoted to the proof of Theorem~\ref{new-boundedness-unified-version}, which we restate here for the reader's convenience.

\begin{thm}
  \label{new-boundedness-unified-version-section}
  Assume that \cref{eq:cond-K} holds with some measurable function $j: \R^N \setminus \{0\} \to [0,\infty)$ satisfying \eqref{eq:cond-A1} for some $\gamma \le 2$, and suppose also that \eqref{eq:cond-K_as2} holds. Moreover, let $\Omega \subset \R^N$ be a bounded open set with continuous boundary, and let $f,W \in L^\infty(\Omega)$. If
  \begin{equation}
    \label{eq:key-comparison-W-j-section}
  \Lambda \|W^+\|_{L^\infty(\Omega)} < \int_{\R^N}j(z)\,dz \le \infty  
  \end{equation}
  then every weak solution $u \in V(\Omega|\R^N)$ of \eqref{general-op-eq} with $u \in L^\infty(\Omega^c)$
satisfies $u \in L^\infty(\Omega)$, and there exists a constant $C>0$ not depending on $f$ and $u$ with
  \begin{equation}
  \label{new-boundedness-unified-version-estimate-section}
\|u\|_{L^\infty(\Omega)} \le C\Bigl( \|f\|_{L^\infty(\Omega)} + \|u\|_{L^\infty(\Omega^c)} +\|u\|_{L^2(\Omega)} \Bigr).    
  \end{equation}
If, in addition, $W \le 0$ and the kernel $K$ is symmetric, the term $\|u\|_{L^2(\Omega)}$ can be dropped in \eqref{new-boundedness-unified-version-estimate-section}. 
\end{thm}

The following proof is related to the $\delta$-decomposition developed in \cite{FJW22}, see also \cite[Theorem 4.1]{FeJa22} for the symmetric case. More precisely, we have to introduce a new variant of this decomposition in order to deal with the weak assumptions of . 

\begin{proof}
  First note that since $u\in V(\Omega|\R^N)$ by the density result \cite[Theorem 3.76]{F20} we also have
$$
\cE(u,w)=\int_{\Omega} (W(x)u +f)w\ dx\quad\text{for all $w\in V_{\Omega}(\R^N)$.}
$$
To show the claim, we use the functions $w_t=(u-t)^+$ with $t \geq \|u\|_{L^{\infty}(\Omega^c)}$, which for such $t$ belong to $V_{\Omega}(\R^N)$ by Lemma~\ref{lem:local-cutoff}.

{\bf Case 1.} We first consider the case where $j \in L^\infty(\R^N)$. Note that with \eqref{eq:cond-A1} it follows $j\in L^p(\R^N)$ for all $p\in[1,\infty]$. We set
$$
c_1(j):= \Lambda^{-1}\|j\|_{L^1(\R^N)}>0 \qquad \text{and}\qquad c_2(j):= \Lambda \Bigl(\|j\|_{L^2(\R^N)}+ \|j\|_{L^1(\R^N)}\Bigr)  >0
$$
and we note that
  \begin{equation}
    \label{eq:key-comparison-W-j-section-variant}
c_1(j) > \|W^+\|_{L^\infty(\Omega)} 
  \end{equation}
by (\ref{eq:key-comparison-W-j-section}). In this case we then have
\begin{align*}
  0 &= \cE(u,w_t)-\int_{\Omega}(f+ Wu)(x) w_t(x)\,dx \\
	&\ge \cE(u,w_t)-  \int_{\Omega}\Bigl(\|f\|_{L^\infty(\Omega)}+\|W^+\|_{L^\infty(\Omega)}u(x)\Bigr)w_t(x)\,dx\\
  &=\int_{\R^N}w_t(x)\biggl(u(x)\Bigl(\int_{\R^N}K(x,y)\,dy-\|W^+\|_{L^\infty(\Omega)}\Bigr) - \int_{\R^N}u(y)K(x,y)\,dy - \|f\|_{L^\infty(\Omega)} \biggr)dx \\
  &\ge \int_{\Omega}w_t(x)\biggl(u(x) \Bigl(\Lambda^{-1}\|j\|_{L^1(\R^N)}-\|W^+\|_{L^\infty(\Omega)}\Bigr) - \Lambda \int_{\R^N}u(y)j(x-y)\,dy - \|f\|_{L^\infty(\Omega)} \Bigr)dx\\
&\ge \int_{\Omega}w_t(x)\biggl(u(x)\Bigl(c(j)-\|W^+\|_{L^\infty(\Omega)}\Bigr)\\
&\qquad - \Lambda \Bigl(\|u\|_{L^2(\Omega)} \|j\|_{L^2(\R^N)}+ \|u\|_{L^\infty(\Omega^c)}\|j\|_{L^1(\R^N)}\Bigr)-\|f\|_{L^\infty(\Omega)} \biggr)dx\\  
&\ge \int_{\Omega}w_t(x)\biggl(u(x)\Bigl(c(j)-\|W^+\|_{L^\infty(\Omega)}\Bigr) - (c_2(j)+1)C_{u,f}\Bigr)\Biggr)dx  
\end{align*}
with
$$
C_{u,f}:= \|f\|_{L^\infty(\Omega)} + \|u\|_{L^2(\Omega)}+ \|u\|_{L^\infty(\Omega^c)}.
$$
For $t \ge \frac{(c_2(j)+1)}{c(j)-\|W^+\|_{L^\infty(\Omega)}}C_{u,f}$, we then have that 
\begin{align*}
  &0 \ge \int_{\Omega}w_t(x)\biggl( \Bigl(c(j)-\|W^+\|_{L^\infty(\Omega)}\Bigr) u(x) - (c_2(j)+1)C_{u,f}\Bigr)\Biggr)dx   \\
  &=
    \int_{\Omega}w_t(x)\biggl(\Bigl(c(j)-\|W^+\|_{L^\infty(\Omega)}\Bigr) (u(x)-t)+t \Bigl(c(j)-\|W^+\|_{L^\infty(\Omega)}\Bigr) -(c_2(j)+1)C_{u,f} \biggr)dx\\
  &\ge \Bigl(c(j)-\|W^+\|_{L^\infty(\Omega)}\Bigr) \|w_t\|_{L^2(\Omega)}^2.
\end{align*}
This implies that $w_t \equiv 0$, i.e., $u\leq t$ in $\Omega$. Repeating the argument with $-u$ and $-f$, we get $u\ge -t$ in $\Omega$. Hence $u \in L^\infty(\Omega)$, and (\ref{new-boundedness-unified-version-estimate-section}) follows.

{\bf Case 2.} Next, in the case where $j \in L^1(\R^N)$ but $j \not \in L^{\infty}(\R^N)$, we perform a kernel decomposition similar to the one in\cite{FJW22,FeJa22}, but we have to split the kernel in a more subtle way than in these papers to deal with our current weak assumptions. For this we define, for every $\tau>0$, the kernel $K_\tau:= \min\{\tau, K\}$ and $K^\tau := K-K_\tau$. We also define
$$
j_\tau:= \min \{\tau, j\} \qquad \text{and}\qquad j^\tau: =(j-\tau)^+,
$$
noting that
$$
\Lambda^{-1}j_\tau(y-x) \le   K_\tau(x,y)\quad \text{and}\quad 
K^\tau(x,y) \le \Lambda j^\tau(y-x)\qquad \text{for every $x,y \in \R^N$.}
$$
We also let $\cE_{\tau}$ resp. $\cE^{\tau}$ be the bilinear forms associated to the kernels $K_{\tau}$ and $K^\tau$, respectively. Similarly as above, we then have that
\begin{align}
&\cE_\tau(u,w_t)=\int_{\R^N}w_t(x)\Bigl(u(x)\int_{\R^N}K_\tau(x,y)\,dy - \int_{\R^N}u(y)K_\tau(x,y)\,dy\Bigr)dx \nonumber \\
&\ge \int_{\Omega}w_t(x)\biggl(u(x) \Lambda^{-1}\|j_\tau\|_{L^1(\R^N)}- \Lambda \Bigl(\|u\|_{L^2(\Omega)} \|j_\tau\|_{L^2(\R^N)}+ \|u\|_{L^\infty(\Omega^c)}\|j_\tau\|_{L^1(\R^N)}\Bigr)\biggr)dx \nonumber\\  
&\ge \int_{\Omega}w_t(x)\biggl( c_1(j,\tau) u(x) - c_2(j,\tau) \Bigl(\|u\|_{L^2(\Omega)}+ \|u\|_{L^\infty(\Omega^c)}\Bigr)\biggr)dx \label{E-upper-tau-eq-0}   
\end{align}
with
$$
c_1(j,\tau):= \Lambda^{-1}\|j_\tau\|_{L^1(\R^N)} \quad \text{and}\quad c_2(j,\tau):= \Lambda \sup_{x\in \Omega} \Bigl(\|j_\tau\|_{L^2(\R^N)}+ \|j_\tau\|_{L^1(\R^N)}\Bigr)  >0.
$$
Moreover, we have
\begin{equation}
\label{E-upper-tau-eq-1}  
\cE^{\tau}(u,w_t)=\cE^{\tau}(u-t,w_t)=\cE^{\tau}(w_t,w_t)-\cE^{\tau}((u-t)^-,w_t) \ge \cE^{\tau}(w_t,w_t),
\end{equation}
since
\begin{align*}
\cE^{\tau}((u-t)^-,w_t)&=\int_{\R^N}\int_{\R^N}\Big((u-t)^-(x)-(u-t)^-(y)\Big)w_t(x)K^{\tau}(x,y)\ dydx\\
                       &=-\int_{\R^N}\int_{\R^N}(u-t)^-(y) w_t(x)K^{\tau}(x,y)\ dydx\leq 0.
\end{align*}
Furthermore, we have 
\begin{align}
  \cE^{\tau}(w_t,w_t)  &= \int_{\R^N}\int_{\R^N}w_t^2(x)K^\tau(x,y)dydx - \int_{\R^N}\int_{\R^N}w_t(x)w_t(y)K^\tau(x,y)dydx \nonumber\\
                       &\ge - \int_{\R^N}\int_{\R^N}w_t(x)w_t(y)K^\tau(x,y)dydx \ge - \Lambda \int_{\R^N}\int_{\R^N}w_t(x) w_t(y)j^\tau(x-y)dydx \nonumber  \\
                      &\ge -\Lambda \|w_t\|_{L^2(\Omega)}^2\|j^\tau\|_{L^1(\R^N)}, \label{E-upper-tau-eq-2}
\end{align}
where we used Young's inequality in the last step. Setting
$$
c_3(j,\tau) := \Lambda \|j^\tau\|_{L^1(\R^N)}
$$
we may combine \eqref{E-upper-tau-eq-0}, \eqref{E-upper-tau-eq-1}, and \eqref{E-upper-tau-eq-2} to see that
\begin{align*}
  \cE(u,w_t)&= \cE_\tau(u,w_t)+\cE^{\tau}(u,w_t) \\
  &\ge  \int_{\Omega}w_t(x)\biggl( c_1(j,\tau) u(x) - c_2(j,\tau) \Bigl(\|u\|_{L^2(\Omega)}+ \|u\|_{L^\infty(\Omega^c)}\Bigr)-c_3(j,\tau)w_t(x) \Biggr)dx  
\end{align*}
Also, since $u$ is a weak solution of \eqref{general-op-eq}, we have, similarly as in Case 1,
$$
\cE(u,w_t) = \int_{\Omega}(f+ Wu)(x) w_t(x)\,dx \le 
 \int_{\Omega}\Bigl(\|W^+\|_{L^\infty(\Omega)}u(x)+ \|f\|_{L^\infty(\Omega)}\Bigr)w_t(x)\,dx  
$$
and therefore 
\begin{align}
  0 \ge& \int_{\Omega}w_t(x)\biggl(\Bigl(c_1(j,\tau)-\|W^+\|_{L^\infty(\Omega)}\Bigr) u(x)-c_2(j,\tau)\Bigl(\|u\|_{L^2(\Omega)}+ \|u\|_{L^\infty(\Omega^c)}\Bigr) \nonumber\\
  &-\|f\|_{L^\infty(\Omega)}-c_3(j,\tau)w_t(x) \biggr)dx \nonumber\\
  &\ge
    \int_{\Omega}w_t(x)\biggl(\Bigl(c_1(j,\tau)-\|W^+\|_{L^\infty(\Omega)} -c_3(j,\tau)\Bigr)w_t(x) +t \Bigl(c_1(j,\tau)-\|W^+\|_{L^\infty(\Omega)}\Bigr) \nonumber\\
  &\qquad -(c_2(j,\tau)+1)C_{u,f} \biggr)dx \nonumber\\
       &= (c_1(j,\tau)-\|W^+\|_{L^\infty(\Omega)}-c_3(j,\tau))\|w_t\|_{L^2(\Omega)}^2 + \int_{\Omega}w_t(x)\biggl(t \Bigl(c_1(j,\tau)-\|W^+\|_{L^\infty(\Omega)}\Bigr) \nonumber\\
  &\qquad -(c_2(j,\tau)+1)C_{u,f} \biggr)dx \label{Case-2-final-est}
\end{align}
Now we note that 
$$
c_1(j,\tau) \to c_1(j) \quad \text{and}\quad c_3(j,\tau) \to 0 \qquad \text{as $\tau \to + \infty$,}
$$
the latter fact being a consequence of the assumption $j \in L^1(\R^N)$. Since (\ref{eq:key-comparison-W-j-section-variant}) still holds in this case, We may fix $\tau>0$ with $c_1(j,\tau)-\|W^+\|_{L^\infty(\Omega)}-c_3(j,\tau)>0$ and then choose $t> \frac{(c_2(j,\tau)+1)}{c_1(j,\tau)-\|W^+\|_{L^\infty(\Omega)}}C_{u,f}$ to conclude from (\ref{Case-2-final-est}) that $w_t \equiv 0$ in this case. As in Case 1, we may repeat the argument with $-u$ and $-f$ to get $u\ge -t$ in $\Omega$. Hence $u \in L^\infty(\Omega)$, and (\ref{new-boundedness-unified-version-estimate-section}) follows.

{\bf Case 3.} Finally, we consider the case where $j \not \in L^1(\R^N)$, i.e., \eqref{eq:cond-A2} holds. 

We have to argue differently in this case and therefore need to adjust the notation. For $\tau>0$, we consider the sets 
$$
M_\tau:= \{z \in \R^N \setminus \{0\}\::\: j_s(z) \le \tau \},\quad 
M^\tau:= \{z \in \R^N \setminus \{0\}\::\: j_s(z) > \tau \} = (\R^N \setminus \{0\}) \setminus M_\tau.
$$
Here, $j_s$ is the symmetric part of $j$ as before, which also satisfies (\ref{eq:cond-A1}) and \eqref{eq:cond-A2}. By definition and (\ref{eq:cond-A1}), we have
\begin{equation}
  \label{eq:monotone-convergence-preliminary}
M_\tau \subset M_{\tau'} \; \text{for $\tau \le \tau'$}\quad \text{and}\quad \bigcup_{\tau>0}M_\tau = \R^{N} \setminus \{0\}\; \text{up to a set of measure zero.}
\end{equation}
Since the sets $M_\tau$ and $M^\tau$ are symmetric, it also follows that
$$
\|j 1_{M_\tau}\|_{L^p(\R^N)} = \|j_s 1_{M_\tau}\|_{L^p(\R^N)}< \infty \qquad \text{for $p \in [1,\infty]$ and every $\tau>0$.}
$$
Moreover,
$$
\lim_{\tau \to \infty} \|j 1_{M_\tau}\|_{L^1(\R^N)} = \infty
$$
by (\ref{eq:monotone-convergence-preliminary}), monotone convergence and \eqref{eq:cond-A2}. Writing $j_\tau = j 1_{M_\tau}$ and $j^\tau = j 1_{M^\tau}$ in the following, we may therefore fix $\tau>0$ such that 
\begin{equation}
  \label{eq:delta-such-that}
 \|j_\tau\|_{L^1(\R^N)}\geq  2\Lambda(W_{\infty}+c),
\end{equation}
where $W_{\infty}:=\|W^+\|_{L^{\infty}(\Omega)}$ and $c$ is the constant in \cref{lemma:garding} corresponding to $2\tilde{A}(K,\tilde k)$ in place of $\tilde{A}(K,\tilde k)$. We now split the kernel $K$ into
$$
K_{\tau}(x,y)=\frac{1}{2} 1_{M_\tau}(x-y)K(x,y)
$$
and
$$
K^{\tau}(x,y)=\Bigl(1_{M^\tau}(x-y) + \frac{1}{2} 1_{M_\tau}(x-y)\Bigr)K(x,y) =  K(x,y)-K_{\tau}(x,y).
$$
We note that $K_\tau$ resp. $K^\tau$ satisfies the same comparison condition (\ref{eq:cond-K}) with $\Lambda>0$ and the comparison functions $\frac{1}{2} j_\tau$ and $j^\tau + \frac{1}{2}j_\tau$, respectively. 
Let again $\cE_{\tau}$ resp. $\cE^\tau$ be the bilinear forms associated with the kernels $K_{\tau}$ and $K^\tau$, respectively.
By the same argument as in Case 2, we have 
\begin{equation}
  \label{eq:upper-tau-lower-est}
\cE^{\tau}(u,w_t) \ge \cE^{\tau}(w_t,w_t).
\end{equation}
Next, we split the kernel $K^\tau$ in its symmetric and antisymmetric parts 
$$
K^\tau_s(x,y):=\frac12 \big( K^\tau(x,y) + K^\tau(y,x) \big),\qquad K^\tau_a(x,y) := \frac12 \big( K^\tau(x,y) - K^\tau(y,x) \big),
$$
and we note that
$$
K^\tau_*(x,y) = \Bigl(1_{M^\tau}(x-y) + \frac{1}{2} 1_{M_\tau}(x-y)\Bigr) K_*(x,y) \qquad \text{for $*=s$ or $*=a$} 
$$
due to the symmetry of the sets $M^\tau$ and $M_\tau$. Consequently, we have
$$
K^\tau_s(x,y)\ge \frac{1}{2}K_s(x,y) \quad \text{and}\quad |K^\tau_a(x,y)| \le |K_a(x,y)|\qquad \text{for $x,y \in \R^N$, $x \not = y$.}
$$
Consequently, condition \eqref{eq:cond-K_as2} is satisfied for $K^\tau$ with the function $\frac{\tilde k}{2}$ in place of $\tilde k$ and
\begin{align*}
  \tilde{A}(K^\tau,\frac{\tilde k}{2}) &= 2 \sup_{x \in \R^N}\int_{\R^N} \frac{\bigl(K^\tau_a(x,y)\bigr)^2}{\tilde{k}(x,y)}1_{\{\tilde{k}(x,y) \not = 0\}}(x,y)\,dy \\
   &\le 2 \sup_{x \in \R^N}\int_{\R^N} \frac{\bigl(K_a(x,y)\bigr)^2}{\tilde{k}(x,y)}1_{\{\tilde{k}(x,y) \not = 0\}}(x,y)\,dy=2\tilde{A}(K,\tilde k).
\end{align*}
By our choice of $c$, we thus get from Lemma~\ref{lemma:garding} and (\ref{eq:upper-tau-lower-est}) the estimate 
\begin{equation}
  \label{eq:upper-tau-lower-est-1}
  \cE^{\tau}(u,w_t) \ge \cE^{\tau}(w_t,w_t) \ge - c \|w_t\|_{L^2(\Omega)}^2.
\end{equation}  
Similarly as in Case 2, we have
\begin{align}
  &\cE_\tau(u,w_t)=\int_{\R^N}w_t(x)\Bigl(u(x)\int_{\R^N}K_\tau(x,y)\,dy - \int_{\R^N}u(y)K_\tau(x,y)\,dy\Bigr)dx \nonumber \\
 &=\frac{1}{2} \int_{\R^N}w_t(x)\Bigl(u(x)\int_{\R^N}1_{M_\tau}(x-y) K(x,y)\,dy - \int_{\R^N}u(y)1_{M_\tau}(x-y) K(x,y)\,dy\Bigr)dx \nonumber \\
&\ge \frac{1}{2}\int_{\Omega}w_t(x)\biggl(u(x) \Lambda^{-1}\|j_\tau\|_{L^1(\R^N)}- \Lambda \Bigl(\|u\|_{L^2(\Omega)} \|j_\tau\|_{L^2(\R^N)}+ \|u\|_{L^\infty(\Omega^c)}\|j_\tau\|_{L^1(\R^N)}\Bigr)\biggr)dx \nonumber\\  
&\ge \frac{1}{2}\int_{\Omega}w_t(x)\biggl( c_1(j,\tau) u(x) - c_2(j,\tau) \Bigl(\|u\|_{L^2(\Omega)}+ \|u\|_{L^\infty(\Omega^c)}\Bigr)\biggr)dx \label{E-upper-tau-eq-0-case-3}   
\end{align}
again with 
$$
c_1(j,\tau):= \Lambda^{-1}\|j_\tau\|_{L^1(\R^N)} \quad \text{and}\quad c_2(j,\tau):= \Lambda \sup_{x\in \Omega} \Bigl(\|j_\tau\|_{L^2(\R^N)}+ \|j_\tau\|_{L^1(\R^N)}\Bigr)  >0.
$$
Combining (\ref{eq:upper-tau-lower-est-1}) and (\ref{E-upper-tau-eq-0-case-3}), we deduce that 
\begin{align*}
  &0 =  \cE(u,w_t)-\int_{\Omega}(W(x) u(x)+f(x))w_t(x)\ dx\\
  &= \cE_\tau(u,w_t)+ \cE^{\tau}(u,w_t)-\int_{\Omega}(W(x) u(x)+f(x))w_t(x)\ dx\\
  &\ge \int_{\Omega}w_t(x)\biggl( \frac{c_1(j,\tau)}{2} u(x) - \frac{c_2(j,\tau)}{2} \Bigl(\|u\|_{L^2(\Omega)}+ \|u\|_{L^\infty(\Omega^c)}\Bigr)\\
  &\qquad \qquad \qquad\qquad \qquad -\|W\|_{L^\infty(\Omega)}u - \|f\|_{L^\infty(\Omega)}- c w_t(x)\biggr)dx\\
&\ge \int_{\Omega}w_t(x)\Bigl( c\bigl(u(x) - w_t(x)\bigr)- \bigl(c_2(j,\tau)+1\bigr)C_{u,f}\Bigr)dx, 
\end{align*}
where we have set 
$$
C_{u,f}:= \|f\|_{L^\infty(\Omega)} + \|u\|_{L^2(\Omega)}+ \|u\|_{L^\infty(\Omega^c)}.
$$
as before and used that $\frac{c_1(j,\tau)}{2} \ge \|W\|_{L^\infty(\Omega)}+c$ by (\ref{eq:delta-such-that}). Since $w_t(x)\bigl(u(x)-w_t(x))\bigr)= t w_t(x)$ for $x \in \Omega$, we thus obtain 
$$
0 \ge \int_{\Omega}w_t(x)\Bigl( c t - \bigl(c_2(j,\tau)+1\bigr)C_{u,f}\Bigr)dx \ge  \int_{\Omega}w_t(x)\,dx
$$
if $t\ge \frac{1+\bigl(c_2(j,\tau)+1\bigr)C_{u,f}}{c}$. For such $t$ we thus conclude that $w_t \equiv 0$ and therefore $u\leq t$ in $\Omega$. Repeating the argument with $-u$ and $-f$, we get $u\ge -t$ in $\Omega$. Hence $u \in L^\infty(\Omega)$, and (\ref{new-boundedness-unified-version-estimate-section}) follows again.

To complete the proof, we finally need to show that the term $\|u\|_{L^2(\Omega)}$ can be dropped in \eqref{new-boundedness-unified-version-estimate-section} if $K$ is symmetric and $W \le 0$ in $\Omega$. For this, we argue by contradiction and assume that the reduced estimate
\begin{equation}
  \label{eq:boundedness-est-section-2}
\|u\|_{L^{\infty}(\Omega)}\leq C\Bigl(\|f\|_{L^{\infty}(\Omega)} +\|u\|_{L^\infty(\Omega^c)} \Bigl).
\end{equation}
is not valid with any constant independent of $f$ and $u$. Then there exists sequences $(f_n)_n \in L^\infty(\Omega)$ and $(u_n)_n \in V(\Omega|\R^N) \cap L^\infty(\Omega^c)$ with the property that, for any $n$, the function 
$u_n$ solves
\begin{equation}
  \label{eq:contradiction-equation-n}
\op u_n = W(x)u_n + f_n \qquad \text{weakly in $\Omega$}
\end{equation}
and
\begin{equation}
  \label{eq:contradiction-est}
\|u_n\|_{L^\infty(\Omega)} \ge n \Bigl(\|u_n\|_{L^\infty(\Omega^c)} + \|f_n\|_{L^\infty(\Omega)}\Bigr)
\end{equation}
Without loss of generality, we may assume $\|u_n\|_{L^{\infty}(\Omega)}=1$ for all $n\in \N$ so that we have
$$
\|u_n\|_{L^\infty(\Omega^c)} + \|f_n\|_{L^\infty(\Omega)} \le \frac{1}{n}\|u_n\|_{L^{\infty}(\Omega)}=\frac1n.
$$
This entails
\begin{equation}
  \label{eq:zero-limit-boundedness2}
\|u_n\|_{L^\infty(\Omega^c)}\leq \frac1n\quad\text{and}\quad   \|f_n\|_{L^\infty(\Omega)}\leq \frac1n.
\end{equation}
Therefore, the function $w_n=(u-\frac1n)^+$ belongs to $V_{\Omega}(\R^N)$ and we have
\begin{align}
\cE(u_n,w_n)&=\int_{\Omega}(W(x)u_n+f_n)w_n\, dx\leq \|f_n\|_{L^{\infty}(\Omega)}\|w_n\|_{L^{\infty}(\Omega)}|\Omega|\notag\\
&\leq \frac1n\Big(\|u_n\|_{L^{\infty}(\Omega)}+\frac1n\Big)|\Omega|\leq \frac2n|\Omega|.\label{eq:zero-limit-boundedness3}
\end{align}
On the other hand, since $K=K_s$ is symmetric (and \eqref{eq:cond-A2} implies \eqref{eq:poincare}), \cref{lemma:poincare} implies
$$
\cE(u_n,w_n)=\cE(u_n-\frac1n,w_n)\geq \cE(w_n,w_n)\geq \Lambda_1(\Omega)\|w_n\|_{L^2(\Omega)}.
$$
The inequality \eqref{eq:zero-limit-boundedness3} thus entails
$$
\lim_{n\to\infty}\|w_n\|_{L^2(\Omega)}=0.
$$
Replacing $u_n$ with $-u_n$ we have
\begin{equation}
  \label{eq:zero-limit-boundedness3-1}
\lim_{n\to\infty}\|u_n\|_{L^2(\Omega)}\leq \lim_{n\to\infty}\Big(\|(u_n-\frac1n)^+\|_{L^2(\Omega)}+\|(u_n-\frac1n)^-\|_{L^2(\Omega)}+\frac{|\Omega|^{\frac12}}{n}\Big)=0.
\end{equation}
On the other hand, (\ref{new-boundedness-unified-version-estimate-section}) combined with \eqref{eq:zero-limit-boundedness2} and \eqref{eq:zero-limit-boundedness3-1} implies
\begin{align}
1=\lim_{n\to\infty}\|u_n\|_{L^{\infty}(\Omega)}\leq \lim_{n\to\infty}C\Big(\|f_n\|_{L^{\infty}(\Omega)}+\|u_n\|_{L^{\infty}(\Omega^c)}+\|u_n\|_{L^2(\Omega)}\Big)=0
\end{align}
A contradiction. Hence (\ref{eq:boundedness-est-section-2}) holds, and the proof is finished.
\end{proof}

\begin{remark}
\label{remark-new-section-boundedness}
  \begin{enumerate}
	\item[(i)] The proof of \cref{new-boundedness-unified-version-section} actually shows the following: Under the assumptions on $K$, $j$, and $\Omega$, if $u\in V(\Omega,\R^N)$ is a function with $u^+\in L^{\infty}(\Omega^c)$ such that $\op u\leq Au+B$ weakly in $\Omega$ for some $A,B\geq0$, then $u^+\in L^{\infty}(\Omega)$ and there is $C>0$, not depending on $B$ or $u$, such that
	$$
	\|u^+\|_{L^{\infty}(\Omega)}\leq C\Big(B+\|u\|_{L^2(\Omega)}+\|u^+\|_{L^{\infty}(\Omega^c)}\Big).
	$$
  \item[(ii)] The assumption $u \in L^\infty(\Omega^c)$ in \cref{new-boundedness-unified-version-section} can be weakened, as it is sufficient to assume that $u \in L^\infty(\Omega' \setminus \Omega)$ for some bounded open set $\Omega' \subset \R^N$ with $\Omega \Subset  \Omega'$, and that $u$ has a bounded $j$-tail in $\Omega'$. 
  In this case, (\ref{new-boundedness-unified-version-estimate-section}) needs to be replaced by
\begin{equation}
  \label{eq:boundedness-est-section-1-general}
\|u\|_{L^\infty(\Omega)} \le C\Bigl( \|f\|_{L^\infty(\Omega)} + \|u\|_{L^\infty(\Omega' \setminus \Omega)} + \|u\|_{L^2(\Omega)} +T(u,\Omega,\Omega')\Bigr)    
  \end{equation}
  with $C>0$ depending also on $\Omega'$. Here, if the kernel $K$ is symmetric and $W\leq 0$, then (A2) can be replaced by the weaker assumption \eqref{eq:poincare} and the term $\|u\|_{L^2(\Omega)}$ can be dropped. The fact that (A2) can be replaced by \eqref{eq:poincare} follows from \eqref{eq:delta-such-that} noting that the right-hand side in this case is $0$ and $\delta$ can be chosen arbitrary, while the Poincar\'e inequality remains true.\\
   The generalization concerning the boundedness of $u$ in $\Omega^c$ can be derived a posteriori from \cref{new-boundedness-unified-version-section}.  Indeed, starting from a solution $u \in L^\infty(\Omega' \setminus \Omega) \cap V(\Omega|\R^N)$ of \eqref{general-op-eq} satisfying $T(u,\Omega,\Omega')<\infty$, we may consider an arbitrary extension $\tilde u \in L^\infty(\R^N)$ of $u\bigl|_{\Omega'}$ with $\|\tilde u\|_{L^\infty(\Omega^c)}=\|u\|_{L^\infty(\Omega' \setminus \Omega)}$. A straightforward estimate then shows that $\tilde u \in V(\Omega|\R^N)$, and $\tilde u$ is a weak solution of $\op \tilde u = W(x) \tilde u + f + \tilde f$ in $\Omega$ with
  $$
  \tilde f(x) = \int_{\R^N \setminus \Omega'}(u(y)-\tilde u(y))K(x,y)\,dy \qquad \text{for $x \in \Omega$.}
  $$
  Since, by (\ref{eq:cond-K}), we have 
  $$
  \|\tilde f\|_{L^\infty(\Omega)} \le \Lambda \bigl( T(u,\Omega,\Omega')+T(\tilde u,\Omega,\Omega')\bigr)
  $$
  and
  $$
  T(\tilde u,\Omega,\Omega') \le \tilde C \|\tilde u\|_{L^\infty(\R^N)} \qquad \text{with}\quad \tilde C= \sup_{x \in \Omega}\int_{\R^N \setminus \Omega'}j(y-x)\,dy,
  $$
  we get \eqref{eq:boundedness-est-section-1-general} by applying \eqref{new-boundedness-unified-version-estimate-section} to $\tilde u$ and $f+ \tilde f$.
\item[(iii)] If in the situation of Theorem~\ref{new-boundedness-unified-version-section} the function $u$ satisfies a priori
$$
\cE(u,v)=\int_{\Omega} (c(x)u(x)+f(x))v(x)\ dx\quad\text{for all $v\in V_{\Omega}(\R^N)$,}
$$
then $\Omega$ can be an arbitrary open bounded set.
\end{enumerate}
\end{remark}

\section{The growth lemma} \label{sec:growth-lemma}

The main tool in the proof of \cref{thm:continuity} is a growth lemma, which we establish in this section. Our version is an adaptation of the one in \cite{Sil06} to the framework of variational solutions. The idea to replace the Lebesgue measure by the Borel measure $\mu_j$ on $\R^N$ defined through
$$
\mu_j(A) = \int_{A} j(z)\,dz \qquad (A \subset \R^N \text{ a Borel set} ) \,. 
$$   
has been developed in \cite{KaMi17}. Throughout this section we assume \eqref{eq:cond-K}, \eqref{eq:cond-K_as2}, and $\gamma\in(0,1)$.

\begin{thm}[Growth Lemma] \label{thm:growth-lemma} Assume $\cref{eq:cond-A1}, \cref{eq:cond-A2}$. Assume that either \cref{eq:cond-A-3-1} or \cref{eq:cond-A-3-2} holds for some $R_0>0$. 
	Then there exist constants $\theta,\eta \in (0,1)$ and a continuous and non-increasing function 
	$$
	h: (0,R_0) \to (0,\infty) \qquad \text{with}\quad \lim_{r \to 0}h(r)=\infty
	$$
	with the following property: For all $R \in (0,R_0)$, $v_\infty > 0$ there exists $r \in (0,R)$ such that for every $v \in V_{\loc}(B_R | \R^N) \cap L^\infty(\R^N)$ satisfying 
	\begin{alignat}{2}
		\|v\|_{L^\infty(\R^N)} &\le v_\infty \quad &&\text{ in }  \R^N  \label{eq:cond-v-1} \\  
		v \le 1 \text{ in } B_{R}  \quad \text{and}\quad  \op  v &\le h(R)  \quad &&\text{ weakly in } B_{R}  \label{eq:cond-v-2}
	\end{alignat}
	and
	\begin{align}
		\label{eq:key-estimate}
		\mu_{j}\Bigl((B_{R} \setminus B_{r}) \cap \{v \le 0\}\Bigr) \ge \frac{1}{2} 
		\mu_{j}(B_{R} \setminus B_{r})
	\end{align}
	we have 
	$$
	v \le 1 -\theta \qquad \text{almost everywhere in } B_{\eta r} \,.
	$$
\end{thm}

This section is devoted to the proof of \cref{thm:growth-lemma}. The first step of the proof is to derive certain key estimates for the quantities 
\begin{align}
	L(r, R):= \int_{B_R \setminus B_r} j(z)\,dz, \qquad  L(r):= \frac{1}{r} \int_{B_r} |z| j(z)\,dz + L(r,\infty) \label{eq:def-L-func}
\end{align}
which are defined for $r>0$ and $R>r$. Recall that we consider a kernel $j$ satisfying conditions \cref{eq:cond-A1} and \cref{eq:cond-A2}. Note that the function $r \mapsto L(r)$ is non-increasing since for $0 < r \leq s $ we have
\begin{align*}
	L(s) &= \frac{1}{s} \int_{B_r} |z| j(z)\,dz  + \frac{1}{s} \int_{B_s \setminus B_r} |z| j(z)\,dz + L(s,+\infty) \\
	&\leq \frac{1}{r} \int_{B_r} |z| j(z)\,dz  +   \int_{B_s \setminus B_r} j(z)\,dz + L(s,+\infty) = L(r) \,.
\end{align*}

The following proposition is a key tool in the proof of \cref{thm:growth-lemma}. To state this proposition we first note that \eqref{eq:cond-A1} implies that 
	\begin{align} \label{eq:cond-B}\tag{B}  \int_{B_r}|z| j(z) d z = \cO (r^\alpha) \quad \text{ as $r \to 0\,$ with $\alpha = 1-\gamma$.} 
	\end{align}

\begin{prop}\label{prop:key-estimate}
	Assume \eqref{eq:cond-A1}, \eqref{eq:cond-A2}. Let $K_0> \frac{2}{2^\alpha-1}$,
	where $\alpha>0$ is as in \eqref{eq:cond-B}.  Then there exists a strictly decreasing sequence of radii $r_k>0$, $k \in \N$, with $\lim \limits_{k \to \infty}r_k =0$ and the following property: For every $R >0$ there exists $k_0=k_0(R,K_0) \in \N$ such that 
	\begin{align}
	\label{eq:r-k-est}
	r_k < R \quad \text{and}\quad L(\vartheta r_k)\leq 2\Bigl( \frac{K_0}{\vartheta}+1 \Bigr) L(r_k,R) \qquad \text{for $k \ge k_0$, $0 < \vartheta \leq 1$.}
	\end{align}
\end{prop}

\begin{proof}
	Let $\alpha>0$ be as in assumption \cref{eq:cond-B} and $K_0$ any number satisfying $K_0 > \frac{2}{2^\alpha-1}$. Set $m(r) = \displaystyle \int_{B_r}|z| j(z)dz$ for $r>0$. 
	First, we claim that there exists a strictly decreasing sequence of radii $r_k>0$, $k \in \N$, with $\lim \limits_{k \to \infty}r_k =0$ and  
	\begin{align}
	\label{eq:r-k-est-simple-1}
	m(r_k) \leq  K_0\, r_k L(r_k, 2r_k) \qquad \text{for all $k \in \N$} \,
	\end{align}
	implying that, for every $k \in \N$ and every $\vartheta \in (0,1]$,
	\begin{align}\label{eq:r-k-est-simple-L}
	L(\vartheta r_k, r_k) = \int\limits_{B_{r_k} \setminus B_{\vartheta r_k} }  j(z) d z \leq  \frac{1}{\vartheta r_k} \int\limits_{B_{r_k} \setminus B_{\vartheta r_k} } |z| j(z) d z  \leq \frac{K_0}{\vartheta} \, L(r_k, 2r_k)\,. 
	\end{align}
	
	Suppose by contradiction that \cref{eq:r-k-est-simple-1} does not hold true. Then there exists $r_*>0$ such that for every $r \in (0,r_*]$  
	\begin{align}
	\label{r-star-est-1}
	m(r)>  K_0\, r L(r, 2r) \,.
	\end{align}
	
	From (\ref{r-star-est-1}) we deduce 
	$$
	m(r) = m(2r)- \int_{B_{2r} \setminus B_{r}}|z| j(z)d z \ge 
	m(2r)- 2r L(r, 2r) \ge m(2r)- \frac{2}{K_0} m(r)\quad \text{for $r \in (0,r_*]$}
	$$
	and thus 
	\begin{align}
	\label{eq:r-est-combined}
	m(r) \ge \Bigl(1+ \frac{2}{K_0}\Bigr)^{-1} m(2r) \quad \text{for $r \in (0,r_*]$.}
	\end{align}

	Inductively, it thus follows 
	$$
	m(2^{-n}r_*) \ge \Bigl(1+ \frac{2}{K_0}\Bigr)^{-n}m(r_*) \,,
	$$
	whereas by condition \eqref{eq:cond-B} we have 
	$$
	m(2^{-n}r_*)  = \cO (2^{-\alpha n}) \qquad \text{as $n \to \infty$.}
	$$
	Since $m(r_*)>0$ by assumptions \eqref{eq:cond-A1}, \eqref{eq:cond-A2} and $1+ \frac{2}{K_0} < 2^\alpha$ by our choice of $K_0$, we obtain a contradiction. Hence 
	\eqref{eq:r-k-est-simple-1} holds. 
	
	Next, let $R>0$ be given. Since $\int_{B_R}j(z)dz = \infty$ by assumptions \eqref{eq:cond-A1} and \eqref{eq:cond-A2}, there exists 
	$r^* \in (0, R)$ such that $L(R, +\infty) \le L(r, R)$ for $r \le r^*$ and therefore 
	\begin{align}
	\label{eq:simple-est-2}
	L(r, +\infty) \le 2L(r, R).\qquad \text{for $r \le r^*$.}
	\end{align}
	
	Next, let $k_0= k_0(R,K_0) \in \N$ such that $r_k \le \min \{r^*,\frac{R}{2}\}$ for $k \ge k_0$. Combining \eqref{eq:r-k-est-simple-1},  \eqref{eq:r-k-est-simple-L}, and \eqref{eq:simple-est-2}, we conclude
	\begin{align*}
	L(\vartheta r_k) &= \frac{1}{\vartheta r_k}  m(\vartheta r_k)+ L(\vartheta r_k, +\infty)
	\le \frac{1}{\vartheta r_k} m(r_k)+ L(\vartheta r_k, r_k)+ L(r_k, +\infty)\\
	&\le \frac{2K_0}{\vartheta}  L(r_k, 2r_k) + 2 L(r_k, R)\le 2\Bigl(  \frac{K_0}{\vartheta}+1 \Bigr)L(r_k, R)\qquad \text{for $k \ge k_0$,}
	\end{align*}
	as claimed. 
\end{proof}

The following proposition is used in the proof of \cref{thm:growth-lemma} when working with condition \eqref{eq:cond-A-3-2}. 

\begin{prop}\label{key-assumption-estimate-2}
	Assume \eqref{eq:cond-A1}, \eqref{eq:cond-A2} and \eqref{eq:cond-A-3-2} for some $R_0>0$. Let $\alpha>0$ resp. $M>0$ as in \cref{eq:cond-B} resp. \cref{eq:cond-A-3-2}. Let $K_0$ be a number with $K_0> \frac{8}{2^\alpha-1}$ and let $(r_k)_k$ be the sequence from \cref{prop:key-estimate}.  Then there exists $\vartheta_0= \vartheta_0(K_0,M) \in (0,1)$ and, for every $R \in (0,R_0)$, a number $k_0= k_0(R_0, R,K_0) \in \N$ such that 
	\begin{align}
	\label{r-k-variation-est}
	\sup_{x\in B_{\vartheta r_k }}\:\int_{B_R\setminus B_{r_k}}|j(z+x)-j(z)|\ dz \le \frac{1}{4} L(r_k, R)
	\quad \text{for $k \ge k_0$, ${\vartheta} \in (0,\vartheta_0]$.}
	\end{align}
\end{prop}

\begin{proof}
	By \eqref{eq:cond-A-3-2} and \eqref{eq:bv-loc-key-property}, we have that, for $0<r<\min \{R,R_0-R\}$ and $\vartheta \in (0,1)$, 
	\begin{align}
	\sup_{x\in B_{\vartheta r}}\int_{B_R \setminus B_{r}}&|j(z+x)-j(z)|\ dz \leq \vartheta r
	\bigl\|Dj\bigr\| \bigl(B_{R+\vartheta r} \setminus \overline{B_{(1-\vartheta)r}}\bigr)\leq \frac{\vartheta r\, M}{(1-\vartheta)r}  \int_{B_{(1-\vartheta)r}^c}j(z)\ dz
	\nonumber\\
	&= \frac{\vartheta M}{1-\vartheta}  L((1-\vartheta)r, +\infty) \le \frac{\vartheta M}{1-\vartheta}  L((1-\vartheta)r)\label{prop-2-variation-est-1}
	\end{align}
	By \cref{prop:key-estimate} we also have 
	\begin{align}
	\label{prop-2-variation-est-2}
	L((1-\vartheta)r_k) \le  \Bigl(\frac{ 2K_0}{1-\vartheta}+2 \Bigr) L(r_k, R) \qquad \text{for $k \ge k_0=k_0(R,K_0)$.}
	\end{align}
	Making $k_0$ larger if necessary, we may also assume that $r_k \le \min \{R,R_0-R\}$ for $k \ge k_0$. Then we may combine \eqref{prop-2-variation-est-1} and \eqref{prop-2-variation-est-2} to get 
	$$
	\sup_{x\in B_{\vartheta r_k}}\int_{B_R \setminus B_{r_k}}|j(z+x)-j(z)|\ dz \le \frac{\vartheta M}{1-\vartheta} \Bigl(\frac{ 2K_0}{1-\vartheta}+2 \Bigr) L(r_k, R) \qquad \text{for $k \ge k_0$.}
	$$
	Choosing $\vartheta_0 \in (0,1)$ small enough such that $\frac{\vartheta M}{1-\vartheta} \Bigl(\frac{ 2K_0}{1-\vartheta}+2 \Bigr) \le \frac{1}{4}$ for $\vartheta \in (0,\vartheta_0]$, we thus get the assertion.
\end{proof}

We now have all tools at hand in order to prove \cref{thm:growth-lemma}.

\begin{proof}[Proof of \cref{thm:growth-lemma}]
	The proof follows the general strategy of \cite{Sil06} and \cite[Lemma 2]{KaMi17}, but new ingredients are needed to deal with weak solutions and the new condition \eqref{eq:cond-A-3-2}.

 Note that, due to \cref{eq:cond-A2}, $\lim_{r\to 0^+}L(r)=\infty$ with $L$ as in \eqref{eq:def-L-func}. Choose a smooth and strictly decreasing function $\beta: [0,\infty) \to [0,\infty)$ satisfying $\beta(0)=1$ and 
	$$
	c_b:= 2 \bigl \|b \bigr\|_{C^1_b(\R^N)} < \infty \qquad \text{for the function $\quad b: \R^N \to \R$, $\quad b(x)= \beta(|x|)$.}
	$$
	Moreover, we let 
	$$
	\beta_r(s) := \beta(\tfrac{s}{r}) \quad \text{and}\quad 
	b_r(x) := \beta_r(|x|)= b(\tfrac{x}{r})  \qquad \text{for $r>0$, $s \ge 0$ and $x \in \R^d$.}
	$$ 
	For $a>1$ we define
	\begin{align}
	\label{eq:def-theta-a}
	\theta_a := \frac{1}{a}\bigl(\beta(\tfrac{1}{2})- \beta(1)\bigr) \in (0,1) \qquad \text{and} \qquad d_a:= 1+ \frac{\beta(1)}{a} = 1-\theta_a + \frac{\beta(\frac{1}{2})}{a} \in (0,2) \,.
	\end{align}
	Moreover, we fix $K_0> \frac{8}{2^\alpha-1}$, where $\alpha$ satisfies \cref{eq:cond-B}, and we consider the corresponding strictly decreasing sequence of radii $r_k>0$, $k \in \N$ with $\lim \limits_{k \to \infty}r_k =0$ given by \cref{prop:key-estimate}. Let us complete the proof in the first case, i.e., under condition \eqref{eq:cond-A-3-1}.
	{\bf Case 1:} We assume condition \cref{eq:cond-A-3-1}. Let $\sigma, c_0 \in (0,1)$ be the constants from condition \cref{eq:cond-A-3-1}.  Set
	\begin{align}
	\label{eq:def-eta-1}
	\vartheta := \sigma, \qquad \eta := \frac{\sigma}{2} = \frac{\vartheta}{2} \,.
	\end{align}
	We define the function $h$ by 
\begin{align}\label{eq:h-func-case-1}
	r\mapsto h(r):= \frac{c_0 \Lambda^{-1}}{33 \bigl(\frac{  K_0}{\vartheta}+1\bigr)} \inf_{0<s\leq r} L(s) \,.
\end{align}
The constants $\Lambda, c_0$ are as in \eqref{eq:cond-K} resp. as in \eqref{eq:cond-A-3-1}. The constants $\vartheta, K_0$ have already be chosen above. Moreover, we fix $a = \max \bigl \{1, \frac{16\Lambda^2 c_b}{c_0} \bigl(\frac{  K_0}{\vartheta}+1\bigr) \}$. We claim that the assertion of \cref{thm:growth-lemma} holds true for this choice of the values $\eta$ and $\theta := \theta_a$.
	
	\medskip
	
	Let $R \in (0,R_0)$ and $v_\infty>0$ be given. Let $k_0 \in \N$ be such that \eqref{eq:r-k-est} holds. Since $\lim \limits_{k \to \infty}L(r_k,R) = \infty$ by assumption \cref{eq:cond-A2}, we may fix $k \ge k_0$ such that 
	\begin{align}
	\label{eq:condition-j}
	\frac{c}{\Lambda}+L(R-r_k, \infty) < \frac{c_0}{\Lambda^2 4 (v_\infty+2)}\, L(r_k,R),    
	\end{align}
	where $c$ is the constant in the G{\aa}rding inequality, Lemma \ref{lemma:garding}.
	We then set $r:= r_k$ and consider a function $v \in V_{\loc}(B_{R}|\R^N) \cap L^\infty(\R^N)$ satisfying \eqref{eq:cond-v-1}, \eqref{eq:cond-v-2} and \eqref{eq:key-estimate}. We then need to show 
	\begin{align}
	\label{eq:label-v-est-required}
	v \le 1 -\theta_a \qquad \text{a.~e. in $B_{\eta r}$.}
	\end{align}
	To see this, we consider the function 
	\begin{align}
	\label{eq:def-u-function}
	u \in V_{\loc}(B_{R}|\R^N) \cap L^\infty(\R^N),\qquad u(x)= v(x) +\frac{1}{a} b_{{\vartheta {r}}}(x)-{d_a}.
	\end{align}
	Since $v \leq 1$, for $x \in B_R \setminus B_{{\vartheta {r}}}$ we conclude
	$$
	u(x) \le 1 + \frac{1}{a} b_{{\vartheta {r}}}(x)-{d_a} \le 1 + \frac{1}{a} \beta_{{\vartheta {r}}}({\vartheta {r}})-{d_a}=0. 
	$$
	The whole proof is complete if we can show 
	\begin{align}
	\label{eq:main-claim}
	u \le  0 \qquad \text{a.~e. in $B_{{\vartheta {r}}}$}\,.  
	\end{align}
	This proves \eqref{eq:label-v-est-required} because then for a.~e. $x \in B_{\frac{{\vartheta {r}}}{2}}= B_{\eta {r}}$
	$$
	v(x) = u(x) - \frac{1}{a} b_{{\vartheta {r}}}(x) + {d_a} \le {d_a} - \frac{1}{a} \beta_{{\vartheta {r}}}(\frac{{\vartheta {r}}}{2}) = 1-\theta_a \,.
	$$
	In order to prove \eqref{eq:main-claim}, we first note, similarly as in \cite{KaMi17}, for $0 < r < 1$ and $x \in \R^N$
	\begin{align*}
	|\op  b_r(x)| & = \Big| \int_{\R^N}\bigl(b(\frac{x}{r})-b(\frac{x+z}{r})\bigr)K(x,x+z)\,dz \Big| \\
	& \leq \int_{B_r} \Big| b(\frac{x}{r})-b(\frac{x+z}{r}) \Big| K(x,x+z) dz + \int_{\R^N \setminus B_r} \Big| b(\frac{x}{r})-b(\frac{x+z}{r}) \Big| K(x,x+z) dz \\
	&\le \Lambda \Bigl( \|\nabla b_r\|_\infty \int_{B_r} |z| j(z)\,dz + 
	2\| b_r \|_\infty \int_{B_r^c} j(z)\,dz\Bigr) \leq  c_b \Lambda L(r) \,. 
	\end{align*}
	This implies
	\begin{align}
	\left\{
	\begin{aligned}
	 \op  u &\le h(R) + \frac{c_b \Lambda}{a}L({\vartheta {r}})  \qquad \text{in $B_{R}$},\\
	u   &\le 0 \qquad \text{on $B_R \setminus B_{{\vartheta {r}}}$} \,.
	\end{aligned}
	\right. 
	\end{align}
	Moreover, $\|u\|_{L^\infty(\R^N)} \le \|v\|_{L^\infty(\R^N)} + 2$. For points $x \in A:= (B_{R} \setminus B_{{r}}) \cap \{v \le 0\}$ we have 
	$$
	u(x) \le \frac{1}{a} b_{{\vartheta {r}}}(x)-{d_a} = -1  + \frac{1}{a}\Bigl(b_{{\vartheta {r}}}(x)- b_{{\vartheta {r}}}({\vartheta {r}})\Bigr)\le -1 
	$$
	Finally, we note that, since $u \le 0$ on $B_R \setminus B_{{\vartheta {r}}}$, we can use the function 
	$w:=u^+ 1_{B_{{\vartheta {r}}}}$ as a test function by \cref{lem:local-cutoff}. The pointwise estimates of $v$ resp. $u$ then imply the following:
	\begin{align*}
	\Bigl(& h(R) + \frac{\Lambda c_b}{a}L({\vartheta {r}})\Bigr)\int_{B_{{\vartheta {r}}}} u^+\,dx \\
	&\geq \cE(u,u^+1_{B_{{\vartheta {r}}}}) =\cE(u^+1_{B_{{\vartheta {r}}}},u^+1_{B_{{\vartheta {r}}}})-\cE(u^-1_{B_{{\vartheta {r}}}},u^+1_{B_{{\vartheta {r}}}})+\cE(u1_{B^c_{{\vartheta {r}}}},u^+1_{B_{{\vartheta {r}}}})\\
	& \ge -c\int_{B_{\vartheta r}} (u^+(x))^2\ dx - \int_{B_{{\vartheta {r}}}}u^+(x) 	\int_{\R^N \setminus B_{{\vartheta r}}} u(y)K(x,y)dydx,
	\end{align*}
	where we have used the G{\aa}rding inequality. In the last line, of this inequality, we have
	$$
	\int_{B_{\vartheta r}} (u^+(x))^2\ dx\leq (\|v\|_{L^\infty(\R^N)}+2)\int_{B_{\vartheta r}} u^+(x)\ dx
	$$
	and
	\begin{align*}
	\int_{B_{{\vartheta {r}}}}u^+(x) 	&\int_{\R^N \setminus B_{{\vartheta r}}} u(y)K(x,y)dydx\\
	&\le   \int_{B_{{\vartheta {r}}}}u^+(x) \int_{\R^N \setminus B_{{r}}} u(y)K(x,y)dydx \\
	&\le   \int_{B_{{\vartheta {r}}}}u^+(x) \biggl( \int_{B_{R}^c}u(y) K(x,y)dy- \int_{A} K(x,y)dy \biggr)dx \\
	&\le   \int_{B_{{\vartheta r}}}u^+(x) \biggl( \Lambda(\|v\|_{L^\infty(\R^N)}+2)\int_{B_{R}^c} j(y-x) dy- \Lambda^{-1}\int_{A} j(y-x)dy  \biggr)dx\\ 
	&\le   \left((v_\infty+2) \Lambda L(R-{r}, \infty)- \Lambda^{-1}\inf_{x \in B_{\vartheta {r}}} \int_{A} j(y-x) dy    \right) \int_{B_{{\vartheta r}}}u^+(x)dx .  
	\end{align*}
		We thus find
\begin{align*}
	\Bigl(h(R)& + \frac{\Lambda c_b}{a}L({\vartheta {r}})\Bigr)\int_{B_{{\vartheta {r}}}} u^+\,dx \\
	&\geq\left(\Lambda^{-1}\inf_{x \in B_{\vartheta {r}}} \int_{A} j(y-x) dy    -(v_\infty+2)\Big(c+\Lambda L(R-{r}, \infty)\Big) \right) \int_{B_{{\vartheta r}}}u^+(x)dx
	\end{align*}
	Let us prove \eqref{eq:main-claim}. We assume that \eqref{eq:main-claim} does not hold true. We divide in the above inequality by $\int_{B_{{\vartheta r}}}u^+(x)dx$ and obtain 
	\begin{align} \label{growth-1-est-1}  
	h(R) + \frac{\Lambda c_b}{a}L({\vartheta {r}}) \ge \Lambda^{-1}\inf_{x \in B_{\vartheta {r}}} \int_{A} j(y-x) dy  -(v_\infty+2)\Big(c+\Lambda L(R-{r}, \infty)\Big).
	\end{align}
	Let us estimate the right-hand side from below. Since $|y| \ge {r}$ for every $y \in A$, we know
	$$
	1- \vartheta \le   \frac{|x-y|}{|y|} \le  1+ \vartheta \qquad \text{for $y \in A$, $x \in B_{\vartheta {r}}$.}
	$$
	Recalling that $\vartheta = \sigma$ in this case, we deduce from assumption \cref{eq:cond-A-3-1} and \eqref{eq:key-estimate} 
	\begin{equation}\label{growth-1-est-2}  
	\inf_{x \in B_{\vartheta {r}}} \int_{A} j(y-x)dy \ge c_0 \int_A j(y)\,dy = c_0 \mu_j(A) \geq  \frac{c_0}{2} \mu_{j}(B_{R} \setminus B_{r}) = \frac{c_0}{2}L(r,R).
	\end{equation}

	Applying \eqref{eq:condition-j} to \eqref{growth-1-est-1}, \eqref{growth-1-est-2} and afterwards \eqref{eq:r-k-est}  leads to
	\begin{equation*}
	h(R) +  \frac{\Lambda c_b}{a}L({\vartheta {r}})
	\ge  \Lambda^{-1} \Bigl(\frac{c_0}{2} -\frac{c_0}{4} \Bigr) L(r,R) \ge  \frac{c_0 \Lambda^{-1}}{8 \bigl(\frac{  K_0}{\vartheta}+1\bigr)} L({\vartheta {r}})   
	\end{equation*}
	and thus 
	$$
	h(R)\ge \Bigl(\frac{\Lambda^{-1} c_0}{8 \bigl(\frac{  K_0}{\vartheta}+1\bigr)} -\frac{\Lambda c_b}{a}\Bigr)L({\vartheta {r}})
	$$
	Since $a \geq  \frac{16 \Lambda^2 c_b}{c_0} \bigl(\frac{  K_0}{\vartheta}+1\bigr)$, we get  
	$$
	h(R)  \geq \frac{c_0 \Lambda^{-1}}{16 \bigl(\frac{  K_0}{\vartheta}+1\bigr)}  L({\vartheta {r}}) \,.
	$$
	With \eqref{eq:h-func-case-1} it follows $\frac{33}{16}h(R)\leq h(R)$, a contradiction. The assumption that \eqref{eq:main-claim} does not hold true has led to a contradiction. The proof is finished in Case 1.

\medskip

	{\bf Case 2:} Let us now assume condition \cref{eq:cond-A-3-2}.  
 
In this case we put
$$
\vartheta= \vartheta_0\qquad \text{and}\qquad \eta = \frac{\vartheta_0}{2},
$$
where $\vartheta_0= \vartheta(K_0,M)$ is given by \cref{key-assumption-estimate-2}. Moreover, we fix 
$$
a > \max \bigl \{1,  32 \Lambda^2 c_b \bigl(\frac{ K_0}{\vartheta}+1\bigr)  \},
$$
and we claim that the assertion of the theorem holds with these values of $\eta$, $\kappa$ and $\theta := \theta_a$ as given in (\ref{eq:def-theta-a}). To see this, we let $R \in (0,R_0)$ and $v_\infty>0$ be given, and we let $k_0 \in \N$ be such that (\ref{eq:r-k-est}) holds. Since $\lim \limits_{k \to \infty}L(R,r_k) = \infty$ by assumption \eqref{eq:cond-A2}, we may fix $k \ge k_0$ such that 
\begin{align}
	\label{eq:condition-j-1}
	\frac{c}{\Lambda}+L(R-r_k, +\infty) < \frac{1}{8 (v_\infty+2)}\, L(R,r_k),     
\end{align}
where $c$ is the constant from the G{\aa}rding inequality as in the first case. We then put $r:= r_k$ and we consider a function $v \in V_{\loc}(B_{R}|\R^N) \cap L^\infty(\R^N)$ satisfying \eqref{eq:cond-v-1}, \eqref{eq:cond-v-2}, and \eqref{eq:key-estimate}. Again, we need to show (\ref{eq:label-v-est-required}) with the current values of $r$, $a$ and $\eta$. For this we define the function $u$ as in (\ref{eq:def-u-function}), and we need to show that 
\begin{align}
	\label{eq:main-claim-2}
	u \le  0 \qquad \text{a.~e. in $B_{{\vartheta {r}}}$.}  
\end{align}
Assuming by contradiction that (\ref{eq:main-claim-2}) does not hold, we find, as before in \eqref{growth-1-est-1}, that 	
\begin{align}
	\label{growth-1-est-1-2}  
	  h(R) + \frac{\Lambda  c_b}{a}L({\vartheta {r}}) \ge \Lambda^{-1}\inf_{x \in B_{\vartheta {r}}} \int_{A} j(y-x)dy -  (v_\infty+2)\big(c+\Lambda L(R-{r}, +\infty)\big).
\end{align}
Here, $h$ is given by
\begin{equation}\label{eq:h-func-case-2} 
r\mapsto h(r) := \frac{\Lambda^{-1}}{33 \bigl(\frac{K_0}{\vartheta }+1\bigr)} \inf_{0<s\leq r} L(s),
\end{equation}
so that we have as before $\lim_{r\to 0} h(r)=\infty$, and, as in Case 1, $A$ is given by $A= (B_{R} \setminus B_{{r}}) \cap \{v \le 0\}$. Moreover, for $x \in B_{\vartheta {r}}$ we have, by our choice of $\vartheta$, \cref{key-assumption-estimate-2} and \eqref{eq:key-estimate}, 
\begin{align}
	\inf_{x \in B_{\vartheta {r}}} &\int_{A} j(y-x)dy \ge \int_{A} j(y)dy - \sup_{x \in B_{\vartheta {r}}}\: \int_{A}|j(y-x)-j(y)|dy \label{growth-1-est-2-II}\\
	&\ge \int_{A} j(y)dy - \sup_{x \in B_{\vartheta {r}}}\: \int_{B_{R} \setminus B_{{r}}}|j(y-x)-j(y)|dy \ge \Bigl(\frac{1}{2}-\frac{1}{4}\Bigr)L(R,{r}) = \frac{1}{4}L(R,{r}).  \nonumber
\end{align}
Combining (\ref{eq:r-k-est}), (\ref{eq:condition-j-1}), (\ref{growth-1-est-1-2}), and (\ref{growth-1-est-2-II}) now gives   
\begin{equation*}
  h(R) + \frac{\Lambda c_b}{a}L({\vartheta {r}}) 
	\ge \bigl(\frac{1}{4\Lambda}-\frac{1}{8\Lambda}\bigr)   L(R,{r}) \ge \frac{1}{16\Lambda \bigl(\frac{ K_0}{\vartheta}+1\bigr)} L({\vartheta {r}})
\end{equation*}
and thus, since $a > 32 \Lambda^2 c_b \bigl(\frac{ K_0}{\vartheta}+1\bigr)$, we get that 
$$
h(R) \ge \Bigl(\frac{1}{16\Lambda \bigl(\frac{ K_0}{\vartheta}+1\bigr)}-\frac{\Lambda c_b}{a}\Bigr) L({\vartheta {r}})\ge \frac{\Lambda^{-1}}{16 \bigl(\frac{K_0}{\vartheta }+1\bigr)}  L({\vartheta {r}}).
$$
With \eqref{eq:h-func-case-2} it follows $\frac{33}{16}h(R)\leq h(R)$, a contradiction. Again, the assumption that \eqref{eq:main-claim-2} does not hold true has led to a contradiction. Thus, the proof is finished in Case 2 as well and the proof of \cref{thm:growth-lemma} is complete.
\end{proof}

\section{Proof of \texorpdfstring{\cref{thm:continuity}}{Theorem 1.1}}\label{sec:continuity}

In this section we provide the proof of a generalization of \cref{thm:continuity}. The key tool in the proof will be the growth lemma given in \cref{thm:growth-lemma}.

	\begin{thm}
		\label{thm:continuity:section}
		Suppose that \cref{eq:cond-K}, \eqref{eq:cond-K_as2}, \cref{eq:cond-A1} with $\gamma\in(0,1)$, and \cref{eq:cond-A2} hold for a nonnegative function $j \in L^1_{loc}(\R^N \setminus \{0\})$. Assume further that \eqref{eq:cond-A-3-1} \underline{or} \eqref{eq:cond-A-3-2} is satisfied for some $R_0>0$. 
		Moreover, let $\Omega \subset \R^N$ be open and $f, W \in L^\infty_{\loc}(\Omega)$. Then every weak solution $u \in V(\Omega|\R^N) \cap L^{\infty}_{loc}(\Omega)$  of \eqref{general-op-eq} with bounded $j$-tail in $\Omega$ has a continuous representative.
		
		Moreover, for open subsets $A \Subset B_* \Subset B \Subset \Omega$, there exists a nondecreasing continuous function $\omega : [0, \infty) \to [0, \infty)$ independent of $f$ and $u$ with $\omega(0)=0$ and the property that 
		\begin{align}
			\label{new-continuity-main-theorem-estimate:section}
                  &|u(x) - u(y)|\\
                  &\; \leq \omega(|x-y|)\bigl(\|u\|_{L^{\infty}(B)}+ \|f\|_{L^\infty(B_*)} + T(u,B_*,B)\bigr) \qquad \text{for all $x,y \in A$.}\nonumber
		\end{align} 
	\end{thm}

Before giving the proof of this theorem, we briefly explain why \cref{thm:continuity} is an immediate corollary. Indeed, if $u \in L^\infty(\R^N)$ as assumed in 
\cref{thm:continuity}, then $u$ has a bounded $j$-tail in $\Omega$. Moreover, for given open sets $A \Subset B \Subset \Omega$, we may choose an open subset $B_* \subset \Omega$ with $A \Subset B_*  \Subset B$ and obtain (\ref{new-continuity-main-theorem-estimate:section}) with a function $\omega$ having the desired properties. Since 
$$
T(u,B_*,B) \le \tilde C \|\tilde u\|_{L^\infty(\R^N)} \qquad \text{with}\quad \tilde C= \sup_{z \in B_*}\int_{\R^N \setminus B}j(y-z)\,dy
$$
and $\|f\|_{L^\infty(B_*)} \le \|f\|_{L^\infty(B)}$, we thus get (\ref{new-continuity-main-theorem-estimate}) by replacing $\omega$ with $(1+\tilde C)\omega$.

\begin{proof}[Proof of \cref{thm:continuity:section}]   
          It clearly suffices to show (\ref{new-continuity-main-theorem-estimate:section}) for open subsets $A \Subset B_* \Subset B \Subset \Omega$. We put
        $$
        c(f,u):=\|u\|_{L^{\infty}(B)}+ \|f\|_{L^\infty(B_*)} + T(u,B_*,B),
        $$
        and we may assume that $c(f,u)>0$, since otherwise $u \equiv 0$ on $B$ and therefore (\ref{new-continuity-main-theorem-estimate:section}) holds with any function $\omega : [0, \infty) \to [0, \infty)$. We now consider an arbitrary extension $\tilde u \in L^\infty(\R^N)$ of $u \big|_{B}$ satisfying $\|\tilde u \|_{L^\infty(\R^N)} = \|u\|_{L^\infty(B)}$. Then $\tilde u \in V(B_*|\R^N)$, and
        $\tilde u$ is a weak solution of
        $$
        \op \tilde u = \tilde f \qquad \text{in $B_*$}
        $$
        with 
  $$
  \tilde f(x) = W(x) u(x) +f(x) + \int_{\R^N \setminus B}(u(y)-\tilde u(y))K(x,y)\,dy \qquad \text{for $x \in B_*$.}
  $$
  where, by (\ref{eq:cond-K}), we have 
  $$
  \Bigl|\int_{\R^N \setminus B}(u(y)-\tilde u(y))K(x,y)\,dy\Bigr|
\le \Lambda \bigl( T(u,B_*,B)+T(\tilde u,B_*,B)\bigr)
  $$
  for $x \in B_*$ and
  $$
  T(\tilde u,B_*,B) \le \tilde C \|\tilde u\|_{L^\infty(\R^N)} = \tilde C \|u\|_{L^\infty(B)} \quad \text{with}\quad \tilde C= \sup_{z \in B_*}\int_{\R^N \setminus B}j(y-z)\,dy.
  $$
Hence we get that
        \begin{align}
          \|\tilde f\|_{L^\infty(B_*)} &\le \Bigl(\Lambda \tilde C + \|W\|_{L^\infty(B)}\Bigr)\|u\|_{L^\infty(B)}+ \|f\|_{L^\infty(B_*)} +  \Lambda T(u,B_*,B) \nonumber\\
                                       &\le \tilde K\, c(u,f)\qquad \text{with}\quad      \tilde K : = \max\{1,\Lambda, \Lambda \tilde C, \|W\|_{L^\infty(B)}\}.
                                         \label{eq:tilde-f-est}
        \end{align}
        We also note that all quantities in the remainder of the proof will only depend on the values of $u \big|_{B} = \tilde u\big|_{B}$. 
        We consider $R_0>0$ as in the assumption, so $R_0$ also satisfies the assumption of \cref{thm:growth-lemma}. We then put $R_*:=\frac{1}{2} \min\bigl\{R_0, \dist(A,\R^N \setminus B_*)\bigr\}$, and we let $\eta,\theta\in(0,1)$ and $h$ be given by \cref{thm:growth-lemma}. We now first show that for every $x_0 \in A$ and $R \in (0,R_*]$ there is $r'\in(0,R)$ independent of $u$, $f$ and $x_0$ with  
	\begin{align}\label{eq:osc-claim}
	\cO( r') \le \max \Bigl\{\frac{2-\theta}{2}\cO(R),\frac{2 \tilde K c(f,u)}{h(R)}\Bigr\}  \,,
	\end{align}
where, here and in the following, 
	$$
	M_*(r):= \underset{B_r(x_0)}{\essinf} \: u ,\quad M^*(r):= \underset{B_r(x_0)}{\esssup}\: u, \quad \text{and}\quad \cO(r):= \frac{M^*(r) - M_*(r)}{2}
	$$
	for $r \in (0,R_*]$. Clearly, $\cO(r)$ is a nondecreasing function in $r$ which obviously depends on $x_0$, but we do not indicate this dependence in our notation. So let $x_0 \in A$ and $R \in (0,R_*]$ be given, and let us assume
        without loss that $x_0 = 0$. If $\cO(R) \le \frac{2 \tilde K c(f,u)}{h(R)},$ we may pick $r' \in (0,R)$ arbitrarily, and (\ref{eq:osc-claim}) holds. So let us consider the remaining case where 
        \begin{equation}
          \label{eq:oscillation-second-alternative}
         \cO(R)> \frac{2 \tilde K c(f,u)}{h(R)}. 
       \end{equation}
We define the function $v_R \in V_{\loc}(B_R|\R^N) \cap L^\infty(\R^N)$,
	$$
	v_R(x)=\frac{2}{\cO(R)}\Bigl(\tilde u(x) - \frac{M_*(R)+M^*(R)}{2}\Bigr),
	$$
and we note that 
	\begin{align}
	\label{eq:def-kappa-infty-new}
          \|v_R\|_{L^\infty(\R^N)} \le \frac{4}{\cO(R)} \|\tilde u\|_{L^\infty(\R^N)}= \frac{4}{\cO(R)} \|u\|_{L^\infty(B)} \le \frac{4 h(R) \|u\|_{L^\infty(B)}}{2 \tilde K c(f,u)} \le \frac{2}{\tilde K} h(R)
	\end{align}
	and 
	$$
	\underset{B_R}{\essinf}  \: v_R \geq - 1, \qquad  \underset{B_R}{\esssup}\: v_R \leq 1.
	$$
	Since $\op  v_R = \frac{2 \tilde f}{\cO(R)}$ in $B_*$, we also have that 
	$$
	|\op v_R| \le \frac{2 \|\tilde f\|_{L^\infty(B_*)}}{\cO(R)}
 \le \frac{2 \tilde K c(f,u)}{\cO(R)}
        \le h(R) \quad \text{in $B_R(0) \subset B_*$} 
	$$
by (\ref{eq:oscillation-second-alternative}). Next, we choose $r \in (0,R)$ as in \cref{thm:growth-lemma} for $v_\infty= \frac{2}{\tilde K} h(R)$. We may then apply this theorem to $v_R$ or $-v_R$ depending on whether 
	\begin{align}\label{eq:key-estimate-app-1-new}
	\mu_{j}\Bigl((B_{R} \setminus B_{r}) \cap \{v \le 0\}\Bigr) \ge \frac{1}{2} 
	\mu_{j}(B_{R} \setminus B_{r})
	\end{align}
	or 
	\begin{align}\label{eq:key-estimate-app-2-new}
	\mu_{j}\Bigl((B_{R} \setminus B_{r}) \cap \{v \ge 0\}\Bigr) \ge \frac{1}{2} 
	\mu_{j}(B_{R} \setminus B_{r}).
	\end{align}
	\cref{thm:growth-lemma} then yields 
	$$
	\underset{B_{\eta r}}{\esssup}\:  v_R \le 1-\theta \,,
	$$
	or 
	$$
	\underset{B_{\eta r}}{\essinf}\: v_R \ge -(1-\theta) \,,
	$$
	where we recall that $\eta, \theta\in(0,1)$ are independent of $R$ and $v_R$. This implies 
	$$
	\underset{B_{\eta r}}{\esssup}\: u \le  \frac{M_*(R)+M^*(R)}{2} + \frac{1-\theta}{2} \cO(R)
	$$
	or 
	$$
	\underset{B_{\eta r}}{\essinf} \: u \ge \frac{M_*(R)+M^*(R)}{2} - \frac{1-\theta}{2} \cO(R)
	$$
	In both cases, we conclude
	$$
	\cO(r') \le \frac{2-\theta}{2}\cO(R) \,
	$$
	with $r'=\eta r$, as claimed in \eqref{eq:osc-claim}.\\

        Next, we define a sequence $(r_n)_n$ in $(0,R_*]$ inductively as follows. We set $r_0 := R_*$. If $r_n$ is already defined for some $n \ge 0$, we then let $r_{n+1}= r'$, where $r'$ is chosen such that \eqref{eq:osc-claim} holds with $R =r_n$. Applying \eqref{eq:osc-claim} inductively, we find that 
        $$
	\cO( r_n) \le \max \{\kappa^n \cO(R_*),  c(f,u) g_n\} \quad \text{with}\quad g_n:=  2 \tilde K \max_{i=1,\dots,n}\frac{\kappa^{i-1}}{h(r_{n-i})}\quad \text{and}\quad \kappa = \frac{2-\theta}{2}.
        $$
        It is easy to see that $g_n \to 0$ as $n \to \infty$, and $g_n$ is independent of $f$ and $u$. Moreover, since $\cO(R_*) \le 2 \|u\|_{L^\infty(B_*)} \le 2 c(f,u)$, we conclude that with $\tilde g_n:= \max\{\kappa^n,g_n\}$ we have         \begin{equation}
          \label{eq:tild-g-n-ineq}
        \cO( r_n) \le c(f,u) \tilde g_n \qquad \text{for all $n \in \N$,}
        \end{equation}
 and $\tilde g_n \to 0$ as $n \to \infty$, whereas $\tilde g_n$ is also independent of $f$ and $u$. We note that this holds independently of the choice of $x_0 \in A$. We may also assume that $\tilde g_n$ is non-increasing in $n \in \N$, otherwise we replace $\tilde g_n$ by $\max \{\tilde g_m \::\: m \ge n\}$ in the following and (\ref{eq:tild-g-n-ineq}) still holds. 

 We now let $(x_n)_n$ be a dense sequence in $A$ (which exists since $A$ is a separable metric space). Then the set
        $$
        I_{n,m}:=\{x\in A\cap B_{r_n}(x_m)\;:\; u(x)<\underset{B_{r_n}(x_m)}{\essinf}\: u \;\text{ or }\; u(x)>\underset{B_{r_n}(x_m)}{\esssup}\: u\}
$$
has measure zero for every $n$, $m$, and the same holds for
$$
I:=\bigcup_{n,m}I_{n,m}.
$$
In the following, we also let
$$
n(t)= \max \{n \in \N\::\: t \le r_n\} \qquad \text{for $t \in (0,R_*]$.}
$$
For $x,y \in A \setminus I$ with $|x-y|\le r_1$, we then have
$$
n(|x-y|)\ge 1\qquad \text{and}\qquad |x-y| \le r_{n(|x-y|)} <r_{n(|x-y|)-1},
$$
which implies that there exists $m \in \N$ with
$$
|x_m -x|< r_{n(|x-y|)-1}-|x-y| \qquad \text{and therefore}\qquad x,y \in B_{r_{n(|x-y|)-1}}(x_m).
$$
Consequently,
$$
|u(x)-u(y)| \le \cO(r_{n(|x-y|)-1})\le c(f,u) \tilde g_{n(|x-y|)-1}
$$
by (\ref{eq:tild-g-n-ineq}). Hence, defining $\tilde \omega: [0,\infty) \to [0,\infty)$ by
$$
\tilde \omega(t):=
\left\{
\begin{aligned}
  &0, &&\qquad \text{$t=0$;}\\
  &\tilde g_{r_{n(t)-1}}, &&\qquad \text{$t \in (0,r_1]$};\\
  &\max\{2, g_{r_0}\} &&\qquad \text{$t> r_1$};
\end{aligned}
\right.
$$
we see that $\tilde \omega$ is nondecreasing, piecewise constant with $\lim \limits_{t \to 0^+} \tilde \omega(t)=0 = \tilde \omega(0)$, and 
\begin{equation}
  \label{eq:tilde-omega-est}
|u(x)-u(y)| \le \tilde \omega(|x-y|)c(f,u) \qquad \text{for all $x,y \in A \setminus I$.}
\end{equation}
By standard linear interpolation, we may finally replace $\tilde \omega$ by a nondecreasing continuous function $\omega:[0,\infty) \to [0,\infty)$ with $\omega(0)=0$ and the property that (\ref{eq:tilde-omega-est}) holds with $\omega$ in place of $\tilde \omega$. Hence there exists representative for $u|_A$ such that
(\ref{new-continuity-main-theorem-estimate:section}) holds with this choice of $\omega$. In particular, this representative is continuous and therefore unique on the set $A$.  A continuous representative of $u$ on all of $\Omega$ can now be found by considering a sequence of open sets
$$
A_1 \Subset A_2 \Subset A_3 \Subset \dots \quad \text{with}\quad \Omega = \bigcup_{n \in \N}A_n
$$
and letting $\bar u$ be defined on $A_n$ as the unique continuous representative of $u\big|_{A_n}$. By uniqueness of continuous representatives, (\ref{new-continuity-main-theorem-estimate:section}) then still holds for $\bar u$ and any choice of subsets $A \Subset B_* \Subset B \Subset \Omega$.
\end{proof}

\appendix

%

\end{document}